\documentclass{amsart}

\usepackage[utf8]{inputenc}
\usepackage{multicol,graphicx,color}
\usepackage{amsmath,amssymb,amsthm}
\usepackage[english]{babel}
\usepackage{bbm}
\usepackage{mathtools}
\usepackage{xcolor}
\usepackage{soul}
\usepackage{xspace}
\usepackage{algorithm}
\usepackage[noend]{algpseudocode}
\usepackage{comment}
\usepackage{siunitx} 
\usepackage{mathrsfs}
\usepackage{placeins}
\usepackage[hypertexnames=false]{hyperref}
\usepackage{cleveref}

\DeclarePairedDelimiter{\ceil}{\lceil}{\rceil}

\DeclarePairedDelimiter{\abs}{\lvert}{\rvert}%

\newcommand{\Func}{{f}} 
\newcommand{\FuncApprox}{{p}} 
\newcommand{\ChebPoly}{{p}}  
\newcommand{\coeff}{{a}} 
\newcommand{\LinMatrix}{{A}} 
\newcommand{\ConstantVec}{{B}}
\newcommand{\IntWidth}{{\alpha}}
\newcommand{\IntCenter}{{\beta}} 
\newcommand{\LinApproxErr}{{E}} 
\newcommand{\ChebMatrix}{{C}} 
\newcommand{\degree}{{d}} 
\newcommand{\dimension}{{n}} 
\newcommand{\ChebConst}{{\kappa}} 
\newcommand{\ApproxError}{{\varepsilon}}

\newcommand{\RatioReduc}{{\alpha}}
\newcommand{\TotalReduc}{{\sigma}}
\newcommand{\DerivMatrix}{{D}}
\newcommand{\AllDerivMatrices}{\mathscr{D}} 
\newcommand{\MinDetDeriv}{{\omega}} 
 
\newcommand{\MaxFuncDeriv}{{\Delta}} 
\newcommand{\MaxFirstDeriv}{{\Lambda}} 
\newcommand{\UglyProd}{{\psi}} 
\newcommand{\BaseCaseError}{{\hat{\epsilon}}} 
\DeclareMathOperator{\adj}{adj}

\newcommand{\dsand}{\quad\text{ and } \quad} 

\newcommand{\C}{\mathbb{C}}
\newcommand{\R}{\mathbb{R}}
\newcommand{\N}{\mathbb{N}}

\renewcommand{\vec}[1]{{{\mathbf{#1}}}}
\newcommand{\0}{{\vec  0 }}

\renewcommand{\c}{{\vec  c }} 

\newcommand{\e}{{\vec  e }}
\newcommand{\f}{{\vec  f }}

\renewcommand{\k}{{\vec  k }}

\newcommand{\p}{{\vec  p }}

\renewcommand{\v}{{\vec  v }}

\newcommand{\x}{{\vec  x }}
\newcommand{\y}{{\vec  y }}
\newcommand{\z}{{\vec  z }}

\newcommand\norm[1]{\lVert#1\rVert}

\newtheorem{theorem}{Theorem}[section]
\newtheorem{corollary}[theorem]{Corollary}
\newtheorem{lemma}[theorem]{Lemma}

\theoremstyle{definition}
\newtheorem{definition}[theorem]{Definition}

\theoremstyle{remark}

\newtheorem{conjecture}[theorem]{Conjecture}

\title[Solving Multivariable Systems]{Chebyshev Subdivision and Reduction Methods for Solving Multivariable Systems of Equations}
\author[Parkinson, et al.]{Erik Parkinson}
\address{Emergent Trading}
\author[]{Kate Wall}
\address{Tufts University, Department of Mathematics}
\author[]{Jane Slagle}
\address{Tufts University, Department of Computer Science}
\author[]{Daniel Treuhaft}
\address{Brigham Young University, Department of Mathematics}
\author[]{Xander de la Bruere}
\address{Brigham Young University, Department of Mathematics}
\author[]{Samuel Goldrup}
\address{University of Chicago, Department of Economics}
\author[]{Timothy Keith}
\address{Brigham Young University, Department of Mathematics}
\author[]{Peter Call}
\address{University of Texas, Austin, Oden Institute}
\author[]{Tyler J. Jarvis}
\address{Brigham Young University, Department of Mathematics}
\email{Jarvis@math.BYU.edu}

\begin{document}

\maketitle

\begin{abstract}
We present a new algorithm for finding isolated zeros of a system of real-valued functions in a bounded interval in $\R^n$. It uses the Chebyshev proxy method combined with a mixture of subdivision, reduction methods, and elimination checks that leverage special properties of Chebyshev polynomials. We prove the method has quadratic convergence locally near simple zeros of the system.  It also finds all nonsimple zeros, but convergence to those zeros is not guaranteed to be quadratic.   We also analyze the arithmetic complexity and the numerical stability of the algorithm and provide numerical evidence in dimensions up to five that the method is both fast and accurate on a wide range of problems.   
Our tests show that the algorithm outperforms other standard methods on the problem of finding all real zeros in a bounded domain.  Our Python implementation of the algorithm is publicly available at \url{https://github.com/tylerjarvis/RootFinding}.
\end{abstract}


\section{Introduction} \label{section: Introduction}

This paper presents a new algorithm for finding all of the real isolated zeros of a system of smooth functions $f_1,\dots, f_n$ in a compact interval $I$ of the form $I = [a_1,b_1]\times \cdots \times [a_n,b_n] \subset \R^n$.   Our algorithm is an improvement on the Chebyshev proxy method combined with a mixture of subdivision, reduction methods, and elimination checks that leverage special properties of Chebyshev polynomials.  

\subsection{Overview and Relation to the Chebyshev Proxy Method}

The Chebyshev proxy method \cite{BoydBook} for finding real zeros of  a system $f_1,\dots, f_n$ in an interval $I\subset \R^n$ involves finding polynomials $\FuncApprox_1, \dots, \FuncApprox_n$, expressed in the Chebyshev basis, that closely approximate the functions $f_1,\dots, f_n$, and then finds the zeros of the approximating polynomials as approximate zeros of the original system.  The polynomials $\FuncApprox_k$ are proxies for the functions $f_k$ and the zeros of the polynomial system are proxies for the zeros of the original system.

The main improvements that our algorithm makes are 
    \begin{enumerate}
    \item We find bounds $\varepsilon_1,\dots, \varepsilon_n$ for the error 
    \begin{equation}\label{eq:approx-bound}
    \|f_k - \FuncApprox_k\|_{\infty} = \max_{\x\in I} |f_k(\x) - \FuncApprox_k(\x)| < \varepsilon_k 
    \end{equation}
    for each $k \in \{1,\dots, n\}$.
    \item For any family $\FuncApprox_1, \dots, \FuncApprox_n$ of polynomials in the Chebyshev basis approximating a family $f_1,\dots, f_n$ of functions, and for any given error bounds $\varepsilon_k$ satisfying \eqref{eq:approx-bound}, our method 
        \begin{enumerate}
        \item Provides a new, fast-converging algorithm to find all of the isolated zeros $\z_1, \dots, \z_D \in I$ of the system of polynomials $\FuncApprox_1,\dots, \FuncApprox_n$, and 
        \item Also finds a small interval around each $\z_k$ such that any zero of the original system $f_1,\dots, f_n$ must lie in the union of these intervals.
        \end{enumerate}
    \end{enumerate}
Thus we not only find proxies for the zeros of the original system, but also guaranteed bounds on the location of any zeros of the original system.   It is, of course, impossible to guarantee that every interval found this way must contain a root of the original system $f_1,\dots, f_n$, because a function $f_k$ could become very small (within $\varepsilon_k$ of $0$) without actually vanishing.

\subsection{Broad Picture of the Algorithm}

Assuming the functions $f_1,\dots, f_n$ are sufficiently smooth, they can be closely approximated by Chebyshev polynomials $\FuncApprox_1,\dots, \FuncApprox_n$. This approximation can be done rapidly using the fast Fourier transform (FFT) \cite{Trefethen_ApproximationTheory}.  
Our particular implementation requires that the functions $f_1,\dots, f_n$ are real-analytic on the interval $I$, which implies that the approximations converge geometrically.  This allows us, with a little more numerical work, to  construct  a bound $\varepsilon_k$ on the maximum approximation error $\max_{\x\in I} |\FuncApprox_k(\x) - f_k(\x)|\leq \varepsilon_k$ over the  interval $I$ for all $k\in \{1,\dots, n\}$.  If the functions are not analytic, one could still construct an error approximation that accounts for the correspondingly slower rate of convergence.

The approximations $\FuncApprox_1,\dots, \FuncApprox_n$, and error bounds $\varepsilon_1,\dots, \varepsilon_n$, transform the problem of finding the common zeros of $f_1,\dots, f_n$ into that of finding small subintervals of $I$ where every polynomial $\FuncApprox_k$ is within $\varepsilon_k$ of zero. Our implementation also finds all zeros of the polynomial system $\FuncApprox_1,\dots, \FuncApprox_n$.  As mentioned above, this is a variant of the \emph{Chebyshev proxy method} of Boyd (see \cite{BoydBook}), also used in the \emph{Chebfun2} package \cite{chebfun2}.

For a given subinterval of $I$ (or for $I$ itself), the algorithm first performs some elimination checks, which use  properties of the Chebyshev polynomials to try to identify that no zeros can lie in the subinterval.  If those tests fail to eliminate the subinterval, the next step is to apply a reduction method that uses properties of the Chebyshev polynomials to construct a linear approximation of the zero locus to repeatedly shrink the interval in which all the zeros must lie.  This reduction step has quadratic convergence in a neighborhood of a simple zero. This is comparable to the convergence rate of Newton's method around a simple zero, but if there are multiple isolated zeros, our method is guaranteed to find them all, whereas Newton is not.  

If repeated applications of the reduction step do not shrink the interval sufficiently, the interval is subdivided into two or more subintervals and the process is repeated, recursively, until all remaining subintervals are sufficiently small.

The main part of the algorithm, that is, finding approximate common zeros of the Chebyshev polynomials using subdivision, elimination, and reduction methods, could be thought of as a Chebyshev analogue of Mourrain and Pavone's Bernstein polynomial zero finder described in \cite{Mourrain}.  However, the reduction and elimination methods used in their paper for Bernstein polynomials do not work for Chebyshev polynomials. Our reduction and elimination methods are new and rely heavily on special properties of Chebyshev polynomials.   

When all the remaining subintervals are sufficiently small, the algorithm checks for possible duplicates occurring near the boundary of adjacent subintervals by combining any subintervals that touch each other and restarting on the resulting larger interval. 
When complete, it returns all the remaining subintervals.   All real zeros in $I$ of the system $f_1,\dots, f_n$ are guaranteed to lie in the union of these subintervals. For each final subinterval we return the zero we found for the $\FuncApprox_1,\dots, \FuncApprox_n$ as the approximate location of a candidate zero. 

This method is computationally less expensive and numerically at least as accurate as other popular algorithms for finding real zeros in a bounded interval.   Many other algorithms, when applied to this problem, have the disadvantage of searching for all zeros in $\C^n$ instead of only real zeros in the bounded interval $I$; these include homotopy methods like \emph{Bertini} \cite{Bertini}, eigenvalue-based methods like those of M\"oller--Stetter \cite{StetterBook} or \cite{Telen}, and resultant-based methods (as used in \emph{Chebfun2}).  \emph{Chebfun2} uses subdivision to get a good low-degree Chebyshev approximation before using resultants, so it does not find all the complex zeros of the original system, but rather the complex zeros of the various approximations.  Mourrain and Pavone's Berstein-basis algorithm does not suffer from the disadvantage of finding unwanted complex zeros, but only works for polynomials expressed in the Bernstein basis.  As an alternative to our algorithm, it could be worth exploring the feasibility of using the Chebyshev proxy method and then converting from the Chebyshev basis to the Bernstein basis and applying the Mourrain--Pavone algorithm.  This particular basis conversion is fairly well conditioned \cite{Rabbah}, so the conversion should not introduce much additional error. 

\subsection{Why Chebyshev?}

The two main parts of the algorithm are (1.) Polynomial approximation and (2.) Interval elimination and reduction.  If an interval cannot be eliminated or sufficiently reduced, then it is subdivided and the algorithm is applied to the new subintervals.  The Chebyshev basis is important for both approximation and interval elimination and reduction, as we now describe:
\begin{enumerate}
    \item \textbf{Approximation:} The Chebyshev basis is very well suited to fast, accurate approximation of smooth functions on a bounded interval, and to finding a good upper bound on the error of that approximation.  
        \begin{enumerate}
        \item We can use the FFT to make a Chebyshev approximation of degree $d$ in $O(d^n \log(d))$ time, which is significantly cheaper than  approximation in other polynomial bases.
        \item The error of the FFT-based Chebyshev approximation drops off geometrically with the degree $d$ if the function being approximated is sufficiently smooth (real analytic).  
        This means even relatively low-degree Chebyshev approximations can be very accurate.
        \item The coefficients of the series expansion in the Chebyshev basis of a sufficiently smooth function decrease geometrically with degree, which allows us to bound the total error of the approximation. 
        \end{enumerate}
    \item \textbf{Interval Elimination and Reduction:} 
    Given a polynomial in the Chebyshev basis, we can obtain good bounds on the truncation error that results from discarding some of the terms in the polynomial (see Equation~\eqref{eq:CoeffBound}):
        \begin{enumerate}
            \item This gives us several tests to eliminate intervals that cannot contain a zero.
            \item This allows us to make good linear approximations of the zero locus of the polynomial, which leads to both of our interval reduction methods.
        \end{enumerate}
\end{enumerate}
Several other properties of the Chebyshev basis, including  the fact that the basis functions are orthogonal in an appropriate inner product, also play a role in the proof that our interval reduction method shrinks quadratically to each simple zero.  

\subsection{Outline}
The outline of this paper is as follows: We detail our new algorithm in section \ref{section: our_algorithm} and give a proof of its quadratic convergence in Section~\ref{section: convergence}. Sections~\ref{section: complexity} and \ref{section: stability} discuss the computational complexity and stability of the algorithm, respectively. Finally, numerical results on various test suites and comparisons to other rootfinding solvers are presented in Section~\ref{section: numerical_tests}.

\section{Detailed Description of the Algorithm} \label{section: our_algorithm}
Our algorithm has two main steps. In the first step, it accepts a set of  functions $\Func_1,\dots, \Func_n$ from $\mathbb{R}^n$ to $\mathbb{R}$ that are smooth on a rectangular interval $[a_1,b_1]\times \cdots \times [a_n,b_n]\subset \R^n$.  After a change of variables to transform the interval to $[-1,1]^n$, it approximates each function $f$  as a polynomial 
\begin{equation}\label{eq:p_i-expansion}
\FuncApprox({\x}) = \sum_{\substack{\k = (k_1,\dots, k_n)\\ k_j\ge 0}} \coeff_{\k}T_{k_1}(x_1) \dots T_{k_n}(x_n),
\end{equation}
expressed in the Chebyshev basis, where each $T_k(x)$ is the Chebyshev polynomial defined recursively by $T_0(x) = 1$, $T_1(x) = x$,  and $T_{k+1}(x) = 2x T_k(x) - T_{k-1}(x)$ for $k \ge 1$.   Equivalently, the Chebyshev polynomials can be defined by the relation $T_k(\cos(t)) = \cos(k t)$.  We call the products of the form $T_{k_1}(x_1)\cdots T_{k_n}(x_n)$ the \emph{Chebyshev basis elements}.   In addition to approximating each $f_i$, the algorithm also computes an upper bound  $\ApproxError_i$ on the approximation error, satisfying  $|f_i({\x}) - \FuncApprox_i({\x})| \le \ApproxError_i$ for all  ${\x} \in [-1,1]^n$.

The second step is a Chebyshev polynomial solver. Given polynomials $\FuncApprox_i$ expressed in the Chebyshev basis and approximation error bounds $\ApproxError_i$, it returns the common zeros of the approximating polynomials $\FuncApprox_1,\dots, \FuncApprox_n$, along with bounding boxes within which any zeros of the functions $\Func_1,\dots, \Func_n$ must reside. If the error bounds $\ApproxError_i$ are large, the bounding boxes may contain multiple zeros of the system $\{\FuncApprox_i\}$ or no zeros, but any common zero of the $\Func_i$ must be contained in the union of the bounding boxes.

We now describe these two steps in detail.

\subsection{Chebyshev Proxy}

Any smooth function on a compact interval $[a_1,b_1]\times \cdots \times [a_n,b_n]\subset \R^n$ can be well-approximated with Chebyshev approximations of sufficiently high degree; see \cite{Trefethen_ApproximationTheory}.
After a linear change of coordinates to transform the interval into $I = [-1,1]^n$,
there is a fast algorithm \cite[Section 9.5]{ACMEV2}, based on the FFT, for finding the coefficients of the Chebyshev basis for the degree-$d$ polynomial approximation on $I$ by evaluating the function at the $(d+1)^n$ \emph{Chebyshev points}, that is, the points with each coordinate of the form $\cos\left(\frac{j \pi}{d}\right)$ for some $j\in \{0,\dots, d\}$.  This algorithm also works well with a different degree $d_i$ in each dimension, evaluating on the corresponding grid of   $\prod_{i=1}^n (d_i+1)$  Chebyshev points.

We need two things when approximating a function $\Func$.  These are, first, to determine what the degree $d_i$ in each coordinate should be for the polynomial approximation $\FuncApprox$ to give a sufficiently close approximation, and second, to determine a good upper bound $\ApproxError \ge \max_{\x\in I} |\Func(\x) - \FuncApprox|$ for the error of the approximation. Note that every $f_i$ can be approximated with a different amount of numerical precision.  For example, $\sin(x)$ can be computed to 16 digits of precision, but written as $\sin(x+10^5\pi)$ it can only be computed to about 12 digits of precision, because of the loss of precision in the evaluation.  Because the FFT is stable \cite[Section 24.1]{Higham}, this error in function evaluation is the main constraint on the precision of a Chebyshev approximation $\FuncApprox_i$. 
Rather than choose each $\varepsilon_i$ in advance and try to find a degree $d_i$ that will approximate within $\varepsilon_i$, we instead try to find the degree $d_i$ that achieves the smallest possible $\varepsilon_i$ for each $f_i$, within the limits imposed by the evaluation error of each $f_i$.

Our approach to finding the approximation degrees and an upper bound on the error is similar to that used by Boyd in one dimension \cite{Boyd2}. Other methods for computing Chebyshev approximations like those used in \emph{Chebfun2} in two \cite{Chebfun2DPaper} or three dimensions \cite{Chebfun3DPaper}, could also be used. 

\subsubsection{Approximation Degrees}\label{sec:ApproxDegree}

We first compute the numerical degree of the function $\Func_k$ in each coordinate $i$, meaning the degree $d_i$ of the polynomial approximation $\FuncApprox_k$ in coordinate $i$, starting with an initial guess of degree eight.  Temporarily setting the degrees of approximation of all other coordinates to five allows successive approximations to be quick, and numerical experiments seem to show that it gives a sufficient number of interpolation points. For these given choices of degree, use the fast approximation algorithm to get a possible approximation.  If the last five terms of the approximation in coordinate $i$ are not all within a predetermined tolerance of zero (we default to $10^{-10}$ times the maximum function evaluation), double the degree in coordinate $i$  and reapproximate, repeating until the last five terms of the approximation are all sufficiently small. This gives a candidate  approximation degree $d$.

To ensure the resulting approximation is sufficiently accurate, compare this degree-$d$ approximation to the  approximation of degree $2d+1$ and check that the average difference in coefficients is less than the desired tolerance. Note that while Boyd uses degree $2d$ for this check, we use degree $2d+1$ because it makes higher-degree coefficients less likely to alias to the same value.  

The final degree $d_i$ in the current coordinate $i$ is then determined by taking the maximum absolute value of the coefficients of terms with degree at least $\frac{3d}{2}$ (which are assumed to have converged to machine epsilon), doubling it, and then choosing the degree $d_i$ to correspond to the last coefficient having magnitude greater than this.  Repeating the process for each coordinate gives a list of approximation degrees $d_1,\dots, d_n$, which are then used to obtain one final full approximation $\FuncApprox_k$. This is repeated for each function $\Func_k$.

\subsubsection{Bounding Approximation Error}\label{sec:ApproxError}

We also need a bound $\ApproxError_k$  for the 
error of the final approximation $\FuncApprox_k$ of $\Func_k$; that is, we seek a small $\ApproxError_k$ satisfying $\max_{\x\in I} |\FuncApprox_k(\x) - f_k(\x)|\le \ApproxError_k$  We do this by using the fact that Chebyshev approximations converge geometrically (or better) for functions that are real analytic on the interval $I$; see \cite[Theorem 8.2]{Trefethen_ApproximationTheory}.  
First approximate the geometric convergence factor by comparing the largest coefficient and the coefficient corresponding to the final degree. Using this as a bound on the rate of convergence, we compute the infinite sum of the bounding geometric terms corresponding to the coefficients left out of the approximation, and use that sum as an upper bound for the norm of the error of the approximation.

The idea behind bounding the approximation error is most easily seen in one dimension.  
Assume $f(x)$ is real analytic and has degree-$d$ Chebyshev approximation $p(x) = \sum_{k=0}^d a_k T_k(x)$. 
We treat $p(x)$ as a truncation of the Chebyshev series representation $f = \sum_{k=0}^\infty a_k T_k(x)$ of $f$.
We assume that $d$ is large enough so that the coefficients converge geometrically to $0$ at some rate $\rho$. That is, for $k > d$ we have
\[
\abs{a_k} \le \abs{a_d} \rho^{d-k}
\]
 Let $a_m$ be the coefficient of largest absolute value. We estimate $\rho$ as $\rho = \abs{\frac{a_m}{a_d}}^{d-m}$.  Since $\abs{T_k(x)}\le 1$ for all $k$, we can now bound the error $\abs{f(x) - p(x)} = \sum_{k>d} a_k T_k(x)$ of the approximation as 
 \begin{align*}
 \abs{f(x) - p(x)} & \le \sum_{k=d+1}^\infty \abs{a_k} \le \sum_{k=d+1}^\infty \abs{a_d} \rho^{d-k}
   = \frac{\abs{a_d}}{\rho-1}.
 \end{align*}
We take this  $\ApproxError = \frac{\abs{a_d}}{\rho-1}$ as the upper bound for the approximation error.

The computation of the approximation error bound $\ApproxError$ in $n>1$ dimensions is similar to one dimension, but involves some additional bookkeeping and is messier to write out.

\subsubsection{Functions with Large Dynamic Range}

We add one extra step in our approximation algorithm in order to handle functions with a large dynamic range. For some functions a single Chebyshev approximation might not be sufficient to approximate it on the given interval. For example, to approximate $\Func(x) = e^x\sin{x}$ on $[0,500]$, the best approximation error we can hope for, using double precision floating point, is $\sim 10^{200}$ because the function attains a maximum magnitude of $10^{216}$ on the interval. This makes it impossible to find zeros accurately in the part of the interval where $x$ and the function values are small. To remedy this, we include a parameter \texttt{maxIntervalSize}, which we give the default value of $10^{-5}$. After solving a system of polynomials, if any of the resulting bounding intervals  is larger than \texttt{maxIntervalSize} in any dimension, we reapproximate the function on that interval and re-solve the system on that interval.  Algorithm~\ref{alg: ChebProxySolve} gives this in pseudocode.

For the example of $e^x\sin{x}$ above, the bounding intervals found by our algorithm on the right are small enough (less than \texttt{maxIntervalSize}), but the leftmost interval is $[0,471]$ because in that subinterval the function  is numerically indistinguishable from the zero function, at least for purposes of constructing the Chebyshev proxy.  Thus the function must be reapproximated on the interval $[0,471]$  and re-solved.  After re-solving most of the resulting intervals are sufficiently small, but the leftmost interval is $[0,443]$, and so the function must be reapproximated on that interval and re-solved.  This happens repeatedly, where the function generally has good resolution on the right side of each subsequent interval, with zeros within about 30 of the right side being found to good precision; but the left side (everything more than 30 from the right side) being lumped into a single interval, forcing repeated reapproximating and re-solving the leftmost interval until it finally reaches $[0,31]$. Thus each zero is found in an interval where it is within $\sim 30$ of the right hand side of the interval.
Doing this, our method finds each of the $160$ zeros with an accuracy of $10^{-5}$. This accuracy could, of course, be improved with a smaller choice of \texttt{maxIntervalSize} or a more sophisticated method of deciding when to re-solve.

\begin{algorithm}
\begin{algorithmic}[1]
\caption{Chebyshev Proxy Solver}
\label{alg: ChebProxySolve}
\Procedure{\texttt{Cheb\_Proxy\_Solve}}{$\f, \mathcal{I}$}
\Comment{List of functions and interval}
\State $\mathbf{p}, \boldsymbol{\ApproxError} \leftarrow \texttt{Cheb\_Approximate}(\f, \mathcal{I})$
\State $\texttt{roots}, \texttt{intervals} \leftarrow \texttt{Cheb\_Solve}(\mathbf{p}, \boldsymbol{\ApproxError}, \mathcal{I})$
\Comment{Initial Solve}
\For{$r_k, \mathcal{I}_k$ in (\texttt{roots}, \texttt{intervals})}
\Comment{Loop over root/interval pairs}
    \If{$\texttt{Size}(\mathcal{I}_k) > \texttt{maxIntervalSize} \textbf{ and } \mathcal{I}_k$ != $\mathcal{I}$}
    \Comment{Re-solve}
        \State Remove $\mathcal{I}_k$ from \texttt{intervals}
        \State Remove $r_k$ from \texttt{roots}
        \State \texttt{newRoots},\texttt{newIntervals} $ \leftarrow \texttt{Cheb\_Proxy\_Solve}(\f, \mathcal{I}_k)$
        \State $\texttt{roots} \leftarrow \texttt{Concat}(\texttt{newRoots}, \texttt{roots})$
        \State $\texttt{intervals} \leftarrow \texttt{Concat}(\texttt{newIntervals}, \texttt{intervals})$
    \EndIf
\EndFor
\State \Return \texttt{roots}, \texttt{intervals}
\EndProcedure
\end{algorithmic}
\end{algorithm}

\subsection{Chebyshev Solver}\label{sec:Chebyshev-Solver}

The main part of our algorithm is the Chebyshev solver (called \texttt{Cheb\_Solve} in Algorithms~\ref{alg: ChebProxySolve} and \ref{alg: ChebSolve}), which takes approximating polynomials $\FuncApprox_1, \dots, \FuncApprox_n$, expressed in the Chebyshev basis, and corresponding approximation error bounds $\ApproxError_1,\dots, \ApproxError_n$ and returns small subintervals of $I = [-1,1]^n$ in which any common real zeros in $I$ of the functions $\Func_1, \dots, \Func_n$ must lie.  It also returns its best estimate of the zeros of the proxy system $\FuncApprox_1, \dots, \FuncApprox_n$ as candidate zeros for the original system $\Func_1, \dots, \Func_n$.
Pseudocode for this solver is given in  Algorithm~\ref{alg: ChebSolve}. 
 It consists of two main steps, applied recursively: First use some checks to see if we can discard the current search interval, and if not, shrink (reduce) the interval as much as possible. If we can neither discard nor shrink the interval further, subdivide the search interval in some or all of the dimensions. Second, apply a simple linear transformation to express the current polynomials on the new, smaller, search intervals as Chebyshev approximations on the standard interval $[-1,1]^n$ (instead of reapproximating the original functions $\Func_1,\dots, \Func_n$ on the new intervals). 
 
 Recursively call the main solver on the new intervals and polynomials.  Once the resulting intervals are sufficiently small, we have the bounding box. Set $\boldsymbol{\ApproxError}$ to $\mathbf{0}$ and solve down to a point to get the zero of the proxy system $\FuncApprox_1, \dots, \FuncApprox_n$. Return this point along with the bounding box.
 
 Finally, when moving back up the recursion, combine bounding boxes that touch and re-solve the system on the resulting combined intervals.  This handles roots whose true bounding boxes cross the boundaries of distinct search intervals.

\subsubsection{Elimination Checks}\label{sec:exclusion}

Given a system of Chebyshev approximations $\FuncApprox_1 \dots \FuncApprox_n$ and  corresponding error bounds $\ApproxError_1 \dots \ApproxError_n$ on the interval $[-1,1]^n$, we use some simple elimination checks to exclude the existence of common zeros of the original system $\Func_1 \dots \Func_n$ in the interval $[-1,1]^n$.   We use two such checks, which we call the \emph{constant term} check and the \emph{quadratic} check, respectively.   The original system $f_1,\dots, f_n$ and the proxy polynomial system $\FuncApprox_1,\dots, \FuncApprox_n$ are both guaranteed to have no zeros in any subinterval eliminated by any of these checks.  These are called \texttt{ExclusionChecks} in Algorithm~\ref{alg: ChebSolve} (Line~2).  

We assume that each polynomial $\FuncApprox({\x})$ is written as a sum of Chebyshev basis elements $T_{k_1}(x_1)\cdots T_{k_n}(x_n)$ with coefficients $a_{\k}$ as in \eqref{eq:p_i-expansion}.  
For any $x\in [-1,1]$ the term $T_k(x)$ can be written as a cosine, which implies that $\abs{T_{k_1}(x_1)\cdots T_{k_n}(x_n)}\le 1$ for all $\x=(x_1,\dots, x_n)\in [-1,1]^n$ and any $\k=(k_1,\dots, k_n)\in \N^{n}$ (here $\N$ denotes the natural numbers $\{0,1,\dots.\}$, including $0$).  Thus any $\FuncApprox({\x}) = \sum \coeff_{\k}T_{k_1}(x_1) \cdots T_{k_n}(x_n)$, is bounded by 
\begin{equation}\label{eq:preCoeffBound}
\abs{\FuncApprox({\x})} \le \sum \abs{\coeff_{\k}}
\end{equation}
for all ${\x}\in [-1,1]^n$.
This last bound is extremely useful and comes up repeatedly, so we will use the following notation for it.
\begin{definition}\label{def:coef-bound}
    For a polynomial $p$ written as $p=\sum_{\k} \coeff_{\k}T_{k_1}(x_1) \cdots T_{k_n}(x_n)$ in the Chebyshev basis, we define 
    \[
    \texttt{CoeffBound}(\FuncApprox)  = \sum_{\k} \abs{\coeff_{\k}}
    \]
    and rewrite the bound \eqref{eq:preCoeffBound} as 
    \begin{equation}\label{eq:CoeffBound}
           \abs{\FuncApprox({\x})} \le \texttt{CoeffBound}(\FuncApprox).
    \end{equation}
\end{definition}

Both of our exclusion checks leverage the bound \eqref{eq:CoeffBound} 
to determine when one of the polynomials cannot vanish on the interval. They do this by splitting each approximation into $\FuncApprox_i({\x}) = q_i({\x}) + r_i({\x})$, where $q_i$ has low degree, and then showing that $\abs{q_i({\x})} > \texttt{CoeffBound}(r_i) + \ApproxError_i$ on the interval, implying that $\abs{\Func_i({\x})} > 0$. We first use a fast check where $q_i$ is just the constant term $\coeff_\0$ of $\FuncApprox_i$, and then we follow that with a slower but more powerful check where $q_i$ is the ``quadratic'' part of $\FuncApprox_i$, meaning the sum of all terms of the form $\coeff_{\k}T_{k_1}(x_1) \cdots T_{k_n}(x_n)$ with $\|\k\|_1 :=k_1 + \cdots + k_n \le 2$, and $r_i$ is the sum of all terms with $\|\k\|_1 > 2$.  In our performance tests the obvious linear analogue of these checks has not provided enough benefit to be worth the (relatively minimal) computational cost, so we do not use it.

\subsubsection{Interval Reduction}\label{sec:reduction}

When the exclusion checks fail to eliminate an interval, either because there is a zero in the interval or because the checks are not sufficiently powerful, then we use a reduction method to shrink the interval and zoom in on any potential zero.  If there is one isolated zero in the interval, then this method has quadratic convergence to that zero, as shown in Section~\ref{section: convergence}.  The combination of both reduction methods of this section is called \texttt{ReductionMethod} in Algorithm~\ref{lstline:reduction} Line~7. 

We use the following notation when discussing the reduction method and throughout the rest of the paper.
\begin{definition} \label{def: Notation Checks}
    Notation for the reduction methods.  
    \begin{itemize}
        \item $\LinMatrix$: $\dimension \times \dimension$ matrix where $\LinMatrix_{i,j}$ is the coefficient of the linear term in dimension $j$ of $\FuncApprox_i$.  If $\e_j = (0,\dots,0, 1, 0\dots)$ is the index vector with $1$ in the $j$th position and zeros elsewhere, then $\LinMatrix_{i,j} =\coeff_{\e_j}$ in the expansion \eqref{eq:p_i-expansion} of $\FuncApprox_i$.
        \item $\ConstantVec$: 
        $\dimension \times 1$ vector of constant terms. $\ConstantVec_i$ is the coefficient   $\coeff_{\0}$ of the constant term in the $i$th polynomial $\FuncApprox_i$.
        
        \item $\LinApproxErr$: $\dimension \times 1$ vector bound on the total error of Chebyshev approximation and linear approximation combined. If we write $\norm{\k}_1 = \sum{k_i}$, and use the expansion  \eqref{eq:p_i-expansion} of $\FuncApprox_i$, then the $i$th entry of $\LinApproxErr$ is $\LinApproxErr_i \coloneqq \ApproxError_i + \sum_{\norm{\k}_1 \geq 2} \abs{\coeff_{\k}}$  
    \end{itemize}
\end{definition}

\paragraph{First Reduction Method}
The first reduction method iterates through every coordinate (indexed by $j$) and every polynomial $\FuncApprox_i({\x})$ (indexed by $i$) and splits each approximation into $\FuncApprox_i({\x}) = q_i({\x}) + r_i({\x})$ with $q_i({\x}) = \LinMatrix_{ij}x_j + B_i$ the constant plus only the linear term in $x_j$ of the Chebyshev expansion \eqref{eq:p_i-expansion}, while  $r_i(\x)$ includes all the remaining terms (higher order terms and the linear terms in variables $x_k$ with $k\neq j$).  At any zero  $\bar{{\x}} = (\bar{x}_1,\dots, \bar{x}_n)$, we have the bound  $\abs{q_i(\bar{{x}_j})} \le \ApproxError_i + \texttt{CoeffBound}(r_i)$, which gives the bounds
\begin{equation}\label{eq:first-reduction}
\frac{-B_i - \ApproxError_i - \texttt{CoeffBound}(r_i)}{\abs{\LinMatrix_{ij}}} \le \bar{{x}}_j \le \frac{-B_i + \ApproxError_i + \texttt{CoeffBound}(r_i)}{\abs{\LinMatrix_{ij}}}.
\end{equation}
The intersection of the standard interval $[-1,1]^n$ with all of the intervals~\eqref{eq:first-reduction} for all $i$ and $j$ gives the first reduction.

\paragraph{Second Reduction Method}
The second reduction method splits $\FuncApprox_i({\x}) = q_i(\x) + r_i(\x)$ with $q_i({\x}) = \sum_{j=0}^n \LinMatrix_{i,j}x_j + \ConstantVec_i$ containing all of the linear terms (not just the term with one particular $x_j$). Any zero $\bar{{\x}}$ must satisfy $\bar{\x} = \LinMatrix^{-1}({\y} - \ConstantVec)$ for some $\y = (y_1,\dots, y_n)$ with $\abs{y_i} \le \ApproxError_i + \texttt{CoeffBound}(r_i) = \LinApproxErr_i$. This defines a parallelepiped in $\R^n$ in which the zero $\bar{\x}$ must lie.  

Rather than finding the smallest interval that contains this parallelepiped, we just bound each coordinate of the vertices of the parallelepiped as follows. 
Each vertex $\v$ can be written as $\v = \LinMatrix^{-1}({\z} - \ConstantVec)$ for some choice of $\z$ with the $i$th coordinate $z_i \in \{-E_i, E_i\}$. So all zeros must lie in the subinterval centered at $-\LinMatrix^{-1} \ConstantVec$ with width $\sum_{k=1}^n (\abs{\LinMatrix^{-1}_{ik}} \LinApproxErr)$ in coordinate $i$.
This reduction method converges quadratically, as shown in Section~\ref{section: convergence}.

In higher dimensions (say, six and up) it is likely that the overall performance of our  Chebyshev rootfinding method would be improved if the coordinatewise bounding method described above were replaced by an efficient method for finding the smallest interval of the form $[a_1,b_1]\times \cdots \times[a_n,b_n]$ containing the intersection of this parallelepiped with the interval $[-1,1]^n$ using.  This could be done, for example, with a good implementation of the standard algorithm for enumerating all the vertices of a convex polytope defined by intersecting halfspaces.

\subsubsection{Chebyshev Transformation Matrix}\label{sec:transformation}

After any reduction or subdivision, the resulting new subinterval can be rescaled  to the standard interval $[-1,1]^n$ with a linear transformation, and the polynomial approximations $\FuncApprox_1, \dots, \FuncApprox_n$ must be reexpressed  in terms of the standard Chebyshev basis on the standard interval.  This operation is called \texttt{Transform} in Algorithm~\ref{lstline:reapprox} Lines~12--24 and 18.  Note that the new Chebyshev approximations have degrees no larger than the degree of the previous approximations did, and usually smaller degrees, because they approximate the functions in a smaller neighborhood.  Differentiability of the original functions suggests that once we have zoomed in small enough, the approximations will be linear.  Details on how the degree changes with this step are given in Section~\ref{subsection: Tau}.

Over the course of solving a system, we often need thousands of such reexpressions. Generating an entirely new approximation of $\Func_1,\dots, \Func_n$ on each new interval would be slow, as this requires a large number of function evaluations.  To avoid this, observe that rewriting a rescaled one-variable Chebyshev polynomial $T_k(\IntWidth x + \IntCenter)$ in terms of $T_{0}(x), \dots, T_k(x)$ is a linear transformation and can be written in matrix form.  Thus for any one-variable Chebyshev polynomial $\FuncApprox(x) = \sum_{k=0}^d \coeff_k T_k(x)$, we can write $\FuncApprox(\IntWidth x + \IntCenter) = \hat{\FuncApprox}(x) = \sum_{k=0}^d \hat{\coeff}_k T_k(x)$, where $\hat{\coeff} = C(\IntWidth, \IntCenter)a$ for some matrix $C(\IntWidth, \IntCenter)$, which we call the \emph{Chebyshev transformation matrix parameterized by $\IntWidth$ and $\IntCenter$}. For convenience, we refer to $C(\IntWidth, \IntCenter)$ as simply $C$ where the choice of $\IntWidth, \IntCenter$ is understood. Using the matrix $C$ allows us to transform any Chebyshev polynomial to a new interval using only matrix multiplication. In higher dimensions, we apply the transformation to each coordinate sequentially. We note that a similar matrix can be constructed for every polynomial basis.\footnote{Any such matrix can be written as $PDP^{-1}$, where $P$ is the change of basis matrix from the power basis, and $D$ is the transformation matrix for the power basis where the $n$th column is the expansion of $(\alpha + \beta)^n$.} 

The entry $C_{ij}$ of the matrix $C$ is the coefficient of $T_i(x)$ when expanding $T_j(\IntWidth x + \IntCenter)$.  The recursive formula for Chebyshev polynomials allows us to construct $C$ iteratively as follows.  Observe that $T_0(\IntWidth x + \IntCenter) = 1$, and $T_1(\IntWidth x + \IntCenter) = \IntWidth T_1(x) + \IntCenter$, giving a base case of $C_{00} = 1$, $C_{01} = \IntCenter$, and $C_{11} = \IntWidth$. Using the recursion $T_{k+1}(x) = 2xT_k(x) - T_{k-1}(x)$, gives 
\[
T_{k+1}(\IntWidth x + \IntCenter) = 2(\IntWidth x + \IntCenter)\sum_{i=0}^k C_{ik}T_i(x) - \sum_{i=0}^{k-1} C_{i,k-1}T_i(x).
\]
Because $C_{ij} = 0$ for $i > j$, and $2xT_k(x) = T_{k+1}(x) + T_{|k-1|}(x)$, this becomes
\[
T_{k+1}(\IntWidth x + \IntCenter) = \sum_{i=0}^k  2\IntCenter C_{ik}T_i(x) - C_{i,k-1}T_i(x) + \IntWidth C_{ik} (T_{i+1}(x) + T_{\abs{i-1}}(x)).
\]
Lining up all the coefficients of $T_i(x)$ gives
\begin{equation} 
\label{recurr_rel}
    C_{i,k+1} = 2\IntCenter C_{ik} - C_{i,k-1} + \IntWidth(C_{i+1,k} + \eta_i C_{i-1,k}),
\end{equation}
where
\[\eta_i = \begin{cases} 
    0 & i = 0 \\
    2 & i = 1 \\
    1 & \text{otherwise.} 
   \end{cases}
\]
Thus column $k+1$ of $C$ can be computed given just columns $k$ and $k-1$. So in practice we never compute and store all of $C$ simultaneously. Instead, each column is recursively computed from the two columns preceding it, and each column is applied to the transformation before the next column is computed. So each column may be discarded once it has been used to compute the two subsequent columns.  The idea of this algorithm is not unique to the Chebyshev basis---a similar algorithm could be constructed for any basis with a short recursion relation.  

Using the Chebyshev transformation matrix $C$ for the reapproximations in this way is both  fast and numerically well-behaved, as shown in Section~\ref{section: stability}.   We expect that this transformation should be computable for degree-$d$ polynomials in $d\log{d}$ time,  which should speed up the solver considerably for systems of large degree.

\subsubsection{Subdivision}\label{sec:subdivide}

As previously outlined, our Chebyshev solver executes checks to eliminate intervals and reduction methods to shrink intervals, but if it cannot eliminate or shrink the interval further, it subdivides the interval.   This operation is called \texttt{SplitInterval} in Algorithm~\ref{lstline:subdivide} Line 17. Notice that as we subdivide, an effect of ``zooming-in" is that smooth functions can be better approximated by polynomials of lower degree. Thus we expect the required approximation degree to decrease as we recurse.  For details, see Section~\ref{subsection: Tau}.

Some naturally occurring systems of functions on an interval, especially systems designed by humans, may have zeros in special locations, including along the hyperplanes dividing the initial interval exactly in half.  In order to avoid the possibility of subdividing along a hyperplane that contains a zero, the very first time we subdivide the original interval we split slightly off half, using a predetermined random number. Subsequent subdivisions split exactly in half for numerical benefits (multiplication and division by $2$ are especially well behaved in  floating-point arithmetic). 

\subsubsection{Recursion}\label{sec:recursion}

After subdivision, our algorithm  recursively calls itself on each of the resulting subintervals, repeating this process until it has found each zero within the original interval.

\hypertarget{par:base-case}{}
\paragraph{Recursion Base Case}
The recursion needs a base case to determine when it has zoomed in sufficiently on a zero and should stop splitting or reducing the interval. This is referenced in Algorithm~\ref{lstline:basecase} Line~9.
Recall that in the second reduction method, $r_i$ represents the terms of $\ChebPoly_i$ of degree at least $2$. If the reduction methods fail to shrink the search interval because $\texttt{CoeffBound}(r_i)$ is large, we should continue subdividing the interval, since transforming the approximation to a smaller interval will shrink the terms in $r_i(x)$ faster than the linear terms (as explained in Section~\ref{section: convergence}). 
If, however, the reduction methods fail to shrink the search interval because $\ApproxError_i$ is large, we should stop subdividing, since transforming the approximation to smaller subintervals will cause the linear terms to shrink while $\ApproxError_i$ remains unchanged. 

To determine which of these cases holds, we set $r_i = 0$ and rerun the reduction methods. Specifically, we rerun the equations in \ref{sec:reduction} after setting $\texttt{CoeffBound}(r_i)$ to $0$. If the size of the resulting interval is not at least $2.5^n$ times smaller than the current interval, we have reached the base case and stop subdividing. The threshold $2.5^n$ is chosen based on Theorem~\ref{thm: ND Linear Check Convergence Error Bound}, and it seems to give good performance in our numerical testing. 

\hypertarget{par:FinalRoot}{}
\paragraph{Get Final Root} When the algorithm reaches the base case, it has found an interval in which a zero may lie. To find a final point to return as the approximate zero, set $\boldsymbol{\ApproxError} = \0$ and continue to zoom in on the zero as before until the interval converges to a point (a true zero of the proxy system $\FuncApprox_1,\dots, \FuncApprox_n$.  This step is called \texttt{SolveFinalRoot} in Algorithm~\ref{lstline:FinalRoot} Line~10.
This convergence occurs quickly, as the convergence is R-quadratic as shown in \ref{thm:quadratic}.

During this final step of assuming $0$ error, if the algorithm eliminates the entire interval by an exclusion check, 
the interval is \emph{not} discarded, because $\FuncApprox({\x})$ still gets within $\ApproxError$ of $0$, and we cannot way whether $\Func({\x})=0$ or not. We still return the interval, but also raise a warning that it may be spurious. Similarly, if multiple roots are found in this final step, this indicates that there may be a double (or higher degree) zero in the system. We return all roots found and similarly raise a warning about possible duplicate zeros. \hypertarget{par:merge}{}

\paragraph{Merging Intervals}
As the algorithm returns from the recursion, it checks whether any of the returned bounding boxes share a boundary. If so, it takes the smallest interval that contains the touching boxes, and re-solves on that interval to determine if the touching intervals correspond to the same zero or different zeros. If this interval is all of $[-1,1]^n$, then it just combines all the intervals together, marking extras as potential duplicates.

\paragraph{Chebyshev Solver Summary} In summary, the whole algorithm is as follows: for each interval run the exclusion checks to see if it should be discarded. If not, run the reduction methods to try to zoom in on a zero. If the interval shrinks sufficiently, reapproximate and solve on the new interval. If it does not shrink sufficiently, check the base case to see if we should stop and return a zero. If we have not hit the base case, split the interval into subintervals and solve each of those recursively, combining and resolving on resulting intervals that touch.  The full algorithm is given in pseudocode in Algorithm~\ref{alg: ChebSolve}.

\begin{algorithm}
\begin{algorithmic}[1]
\algrenewcommand\algorithmiccomment[2][\footnotesize]{{#1\hfill\(\triangleright\) #2}}
\caption{Chebyshev Solver}
\label{alg: ChebSolve}
\Procedure{\texttt{Cheb\_Solve}}{$\mathbf{p}, \boldsymbol{\ApproxError}, \mathcal{I}$}
\Comment{Polynomials $\p$, Errors $\boldsymbol{\ApproxError}$, and Interval $\mathcal{I}$}
\For{\texttt{Check} in \texttt{ExclusionChecks}}
\Comment{Run the exclusion checks}
    \If{$\texttt{Check}(\mathbf{p}, \boldsymbol{\ApproxError}$)}
    \label{lstline:exclusion}
    \Comment{Sec.~\ref{sec:exclusion}}
        \State \Return [ ], [ ]
        \Comment{Throw out interval}
   \EndIf
\EndFor

\State \texttt{changed} $\leftarrow$ True
\While{\texttt{changed}}
\Comment{Run the reduction method while it works}
    \State $\mathcal{\hat{I}}, \texttt{hitBaseCase} \leftarrow \texttt{ReductionMethod}(\mathbf{p}, \boldsymbol{\ApproxError})$
    \label{lstline:reduction}\Comment{Reduction: Sec.~\ref{sec:reduction}}
    \State \texttt{changed} $\leftarrow$ $\texttt{size}(\mathcal{\hat{I}}) \le 0.99 * \texttt{size}(\mathcal{I})$
    \Comment{Reduction worked?}
    \If{\texttt{hitBaseCase}}
    \label{lstline:basecase}\Comment{Sec.~\ref{sec:recursion}{\P}\hyperlink{par:base-case}{\emph{Base Case}}}
        \State $\texttt{root} \leftarrow \texttt{SolveFinalRoot}({\mathcal{I}})$  
        \label{lstline:FinalRoot}
        \Comment{Sec~\ref{sec:recursion}{\P}\hyperlink{par:FinalRoot}{\emph{Get Final Root}}}
        \State \Return [[\texttt{root}],[${\mathcal{I}}$]]  
        \Comment{Return root and interval}
    \ElsIf{\texttt{changed}}
        \label{lstline:reapprox}\Comment{Rescale: {Sec.~\ref{sec:transformation}}}
        \State $\mathbf{p}, \boldsymbol{\ApproxError} \leftarrow (\texttt{Transform}(\FuncApprox_1, \ApproxError_1, \mathcal{\hat{I}}),\ldots,\texttt{Transform}(\FuncApprox_n, \ApproxError_n, \mathcal{\hat{I}}))$
        \State $\mathcal{I} \leftarrow \mathcal{\hat{I}}$
    \EndIf
\EndWhile
\State \texttt{allIntervals} = \texttt{EmptyList}
\State \texttt{allRoots} = \texttt{EmptyList}
\For{$\mathcal{I}_k$ in \texttt{SplitInterval}($\mathcal{I}$)}
\label{lstline:subdivide}\Comment{Subdivide: Sec.~\ref{sec:subdivide}}
    \State $\mathbf{\hat{p}}, \boldsymbol{\hat{\ApproxError}} \leftarrow (\texttt{Transform}(\FuncApprox_1, \ApproxError_1, \mathcal{\hat{I}}),\ldots,\texttt{Transform}(\FuncApprox_n, \ApproxError_n, \mathcal{\hat{I}}))$
     \label{lstline:reapprox2}\Comment{{Sec.~\ref{sec:transformation}}}
    \State $\texttt{roots},\texttt{intervals} \leftarrow \texttt{Cheb\_Solve}(\hat{p}, \boldsymbol{\hat{\ApproxError}}, \mathcal{I}_k)$
    \label{lstline:recursion}\Comment{Recursive solve}
    \State $\texttt{allRoots} \leftarrow \texttt{Concat}(\texttt{allRoots}, \texttt{roots})$
    \State $\texttt{allIntervals} \leftarrow \texttt{Concat}(\texttt{allIntervals}, \texttt{intervals})$
\EndFor
\State $\texttt{CombineIntervals}(\texttt{allRoots}, \texttt{allIntervals})$
\label{lstline:combine}\Comment{Sec.~\ref{sec:recursion}{\P}\hyperlink{par:merge}{\emph{Merging}}}
\State \Return \texttt{allRoots}, \texttt{allIntervals}
\EndProcedure
\end{algorithmic}
\end{algorithm}

\section{Convergence} \label{section: convergence}

In this section we prove the convergence to a simple zero of the second reduction method (using all of the linear terms), as described in Section~\ref{sec:reduction} is R-quadratic.  For the  case that the approximation errors $\ApproxError_1,\dots, \ApproxError_n$ are positive, we give an upper bound for the size of the limiting interval in Theorem~\ref{thm: ND Linear Check Convergence Error Bound}, which allows us to identify when to stop reducing (the base case of Section~\ref{sec:recursion}\P\hyperlink{par:base-case}{Base Case}).  In this section we do not consider the possibility of additional subdivisions because those only occur if the reduction method fails to shrink, which we show will not occur if the initial interval is sufficiently closes to the simple zero.

We first recall the definition of R-quadratic convergence of a convergent sequence and extend it to a nested sequence of intervals. 
\begin{definition} \label{def:quadratic Convergence}
A sequence $(\x_k)_{k=0}^\infty$ converges \emph{quadratically} to $\x_*$ if there exists a $c>0$ such that the errors $e_k = \|\x_k - \x_*\|$ satisfy
\[
e_{k+1} \le c\, e_{k}^2 \quad \text{for all $k\in \N$.}
\]
  The sequence $(\x_k)_{k=0}^\infty$ converges \emph{R-quadratically} if there is a sequence $(r_k)_{k=0}^\infty$ converging quadratically to $0$ such that the errors $e_k = \|\x_k - \x_*\|$ satisfy 
\[
e_k \le r_k \quad \text{for all $k\in \N$}.
\]

For a decreasing sequence of intervals 
\begin{equation}\label{eq:sequence-of-intervals}
J_0 \supseteq J_1 \subseteq J_2 \supseteq \cdots
\end{equation}
in  $\R^n$, with intersection $\bigcap_{k=0}^\infty J_k = \{\x_*\},$
we take the error $e_k$ of the ``approximation'' $J_k$ to be 
$e_k = \max_{\x\in J_k} \|\x - \x_*\|$, which is bounded above by the diameter of $J_k$.  We say that the sequence \eqref{eq:sequence-of-intervals} converges  quadratically to $\x_*$ if $e_k$ converges quadratically to $0$; and we say the sequence of intervals converges \emph{R-quadratically} to $\x_*$ if the errors $e_k$ converge R-quadratically to $0$.
\end{definition}
Quadratic convergence is a very desirable property of a numerical algorithm, since the number of correct digits essentially doubles with each step.  The most famous example of an algorithm that converges quadratically is Newton's method, prompting claims like ``If it's fast, it must be Newton's method''\cite{Tapia}.

\begin{theorem}\label{thm:quadratic}

Let $J_0$ be an interval containing only one zero $\z$ of the Chebyshev proxy system $\FuncApprox_1,\dots, \FuncApprox_n$, and assume also that $\z$ is a simple zero of the proxy system. If $J_0$ is sufficiently small, the second reduction method, using all of the linear terms (see Section~\ref{sec:reduction}), produces a sequence $J_0 \supseteq J_1 \supseteq J_2 \supseteq \cdots$ that converges R-quadratically to the zero $\z$ if the Chebyshev approximation errors $\ApproxError_1,\dots, \ApproxError_n$ are all $0$.  That is, as a solver of Chebyshev polynomial systems, this reduction method converges R-quadratically.   

If the approximation errors $\ApproxError_i$ are not zero, then the reduction method shrinks not to a point, but to an interval. The convergence to the interval is similar to quadratic convergence, as detailed below.
Writing 
\[
J_k = [a_{k,1},b_{k,1}]\times \cdots \times [a_{k,n},b_{k,n}],
\]
let $\gamma_{k,m} = b_{k,m} - a_{k,m}$ be the width of $J_k$ in coordinate $m$, and let $\BaseCaseError_{i,m}$ be as defined in the \hyperlink{def:basecaseerror}{final bullet} of Definition~\ref{def:quadratic-notation} below. 
For each $m\in\{1,\dots, n\}$ the widths $\gamma_{i,m}$ satisfy the (almost-quadratic convergence) relation
\begin{equation}\label{eq:almost-quadratic-sides}
\gamma_{i+1, m} \le K Q_i^2  + \BaseCaseError_{i,m}\gamma_{i,m} 
\end{equation}
for some constant $K$ and a sequence $(Q_i)_{i=0}^\infty$ that converges quadratically to $0$.
Moreover the sequence $(\BaseCaseError_{i,m}\gamma_{i,m})_{i=0}^\infty$ is monotone decreasing.
\end{theorem}

The idea behind this proof is somewhat similar to that of the convergence of Newton's method. We first show that the coefficients of a Chebyshev polynomial are related to the derivatives of the polynomial, and then use the fact that, as we zoom in on a zero, the higher-order derivatives shrink faster than the first-order ones. Similar to Newton's method, this check requires the Jacobian to be invertible in some neighborhood of $\z$, and the higher-order derivatives to be sufficiently small relative to the first-order derivatives.

The rest of this section is dedicated to proving Theorem~\ref{thm:quadratic}.  First, Section~\ref{sec:cheb-derivatives} establishes some relations between Chebyshev coefficients and their derivatives, and then in Section~\ref{sec:R-quadtratic-proof} we use those relationships to prove Theorem~\ref{thm:quadratic}.

\subsection{Relationship between Chebyshev coefficients and derivatives}\label{sec:cheb-derivatives}

We begin by proving a relation between Chebyshev coefficients and their derivatives. This proof uses the fact that Chebyshev polynomials are orthogonal with respect to the weight function $(1-x^2)^{-\frac12}$. While we only need the result for Chebyshev polynomials, the proof holds for a wide range of orthogonal polynomials.

\begin{theorem} \label{thm:Bound On All Degree Coeffs}
Let $\ChebPoly(x) = \sum_{k=0}^d \coeff_k P_k(x)$, where the degree of each $P_k$ is $k$ and $\{P_i\}$ is a basis of polynomials on $[-1,1]$ that are orthogonal (that is, $\int_{-1}^1 P_k(x) P_m(x) W(x)dx = 0$ when $k \ne m$) with respect to some nonnegative weight function $W(x)$ that is zero on at most a set of measure zero.  Also let $B_k$ be the coefficient of $x^k$ in $P_k$. For each $k\in \{0,\dots, d\}$, there exists $c \in [-1,1]$ such that $\coeff_k  = \frac{1}{k! B_k} \ChebPoly^{(k)}(c)$.
\end{theorem}
\begin{proof}   
The $k$th derivative of $\coeff_k P_k(x)$ is $\coeff_k k! B_k$. Thus we need only show the $k$th derivative of the sum of terms of degree $k+1$ and higher is zero at some point; that is, we must show that for any $g(x) = \sum_{i=k+1}^n \coeff_i P_i(x)$ there exists $c \in [-1,1]$ such that  $g^{(k)}(c) = 0$. 

Assume for the sake of contradiction that $g^{(k)}(c) \ne 0$ for all $c \in [-1,1]$. Without loss of generality let $g^{(k)}(x) > 0$ on $[-1,1]$. Any polynomial $h(x)$ of degree at most $k$ can be written in the $\{P_i\}$ basis as $h(x) = \sum_{i=0}^k b_i P_i(x)$, and orthogonality of the $P_i$ guarantees that $\int_{-1}^1 g(x) h(x) W(x)dx = 0$. 

Because $g^{(k)}(x) \ne 0$ for all $x \in [-1,1]$, the polynomial $g$ can have at most $k$ roots on the interval. So let $h(x)$ be a polynomial of degree at most $k$ that has the same roots as $g$ and the same sign as $g$ on each part of the interval.  This implies that $g(x) h(x) \ge 0$ for all $x \in [-1,1]$.  But $W(x) \ge 0$ and $W,g,h$ are nonzero almost everywhere, which implies that $\int_{-1}^1 g(x) h(x) W(x)dx > 0$, a contradiction.
\end{proof}

\begin{corollary} \label{crly:Bound On All Degree Coeffs Cheb}
Let $\ChebPoly(x) = \sum_{k=0}^d \coeff_k T_k(x)$. There exist  $c_0,\dots, c_d \in [-1,1]$ such that  $a_0 = \ChebPoly(c_0)$, $a_1 = \ChebPoly'(c_1)$, and  $\coeff_k  = \frac{1}{k! 2^{k-1}} \ChebPoly^{(k)}(c_k)$ for all $k \ge 2$.
\end{corollary}

We can extend this result to $n$-dimensional systems by induction.
\begin{theorem} \label{thm:Bound On All Degree Coeffs ND}
Let $\ChebPoly({\x})$ be a polynomial of the form $\ChebPoly({\x}) = \sum \coeff_\k P_{k_1}(x_1) \dots P_{k_n}(x_n)$ for $\{P_k\}$ an orthogonal basis of polynomials with the same conditions as described in Theorem~\ref{thm:Bound On All Degree Coeffs}, and for each $k\in \N$ let $B_k$ denote  the coefficient of the monomial $x^k$ in each $P_k$.  For each  multi-index $\k = (k_1,\dots, k_n) \in \N^n$, there exists $\c\in [-1, 1]^n$ such that $ \coeff_\k  = \left(\prod_{i=1}^n \frac{1}{k_i! B_{k_i}}\right) \frac{\partial^{k_1}}{\partial x_1^{k_1}} \dots \frac{\partial^{k_n}}{\partial x_n^{k_n}} \ChebPoly(\mathbf{c})$.
\end{theorem}

\begin{proof}
We denote any monomial $P_{\nu_1} \dots P_{\nu_n}$ as $P_{\boldsymbol{\nu}}$. We write $\boldsymbol{\nu} \not\ge \k$ if $\nu_i < k_i$ for any $i$; we write $\boldsymbol{\nu} = \k$ if $\nu_i = k_i$ for all $i$; and we write $\boldsymbol{\nu} > \k$ otherwise. If $\boldsymbol{\nu} \not\ge \k$, then $\frac{\partial^{k_1}}{\partial x_1^{k_1}} \dots \frac{\partial^{k_n}}{\partial x_n^{k_n}} P_{\boldsymbol{\nu}} = 0$, and $\frac{\partial^{k_1}}{\partial x_1^{k_1}} \dots \frac{\partial^{k_n}}{\partial x_n^{k_n}} P_\k = \prod_{i=1}^n k_i! B_{k_i}.$ So we need only show that if $g({\x}) = \sum a_{\boldsymbol{\nu}} P_{\boldsymbol{\nu}}$, where $\boldsymbol{\nu} > \k$ for all $P$ in the sum, then there exists some $\mathbf{c} \in [-1,1]^n$ such that  $\frac{\partial^{k_1}}{\partial x_1^{k_1}} \dots \frac{\partial^{k_n}}{\partial x_n^{k_n}} g(\mathbf{c}) = 0$. We prove this by induction on the number $n$ of dimensions. The base case for dimension $n=1$ is given in Theorem~\ref{thm:Bound On All Degree Coeffs}.

In dimension $n>1$, split $g({\x}) = g_1({\x}) + g_2({\x})$ where $g_1({\x})$ contains the $P_{\boldsymbol{\nu}}$ where $\nu_1 = k_1$, and $g_2({\x})$ contains everything else. The derivative $\frac{\partial^{k_1}}{\partial x_1^{k_1}} g_1$ is constant with respect to $x_1$, and so is a polynomial in $n-1$ variables that satisfies the criterion of the inductive step. Thus there exist $c_2, \dots, c_n \in [-1,1]$ such that $\frac{\partial^{k_1}}{\partial x_1^{k_1}} \frac{\partial^{k_2}}{\partial x_2^{k_2}} \dots \frac{\partial^{k_n}}{\partial x_n^{k_n}} g_1(x,c_2,\dots,c_n) = 0$ for all  $x \in [-1,1]$.

Given these $c_2, \dots, c_n$, if $h$ is a one-dimensional polynomial of degree at most $k_1$, then  $\int_a^b g_2(x,c_2, \dots, c_n) h(x) W(x)dx = 0$ by orthogonality. By the same argument as in Theorem~\ref{thm:Bound On All Degree Coeffs}, there exists $c_1 \in [-1,1]$ such that $\frac{\partial^{k_1}}{\partial x_1^{k_1}} g_2(c_1, \dots, c_n) = 0$.
Thus at  $\c = (c_1, c_2, \dots, c_n)$, we have $\frac{\partial^{k_1}}{\partial x_1^{k_1}} \dots \frac{\partial^{k_n}}{\partial x_n^{k_n}} g(\mathbf{c}) = 0$, as required.
\end{proof}

\begin{corollary} \label{crly:Bound On All Degree Coeffs Cheb ND}
Let $\ChebPoly({\x}) = \sum \coeff_\k  T_{k_1}(x_1) \dots T_{k_n}(x_n)$. Let $\mathbf{e}_i = (0,\dots, 1, \dots, 0)$ be the index vector with $1$ in the $i$th position and zeros elsewhere. There exist points  $\mathbf{q},\mathbf{r}_i, \mathbf{s_k} \in [-1,1]^n$ such that the 
coefficients of $\ChebPoly$ are related to partial derivatives evaluated at these points as follows:   $\coeff_{\0} = \ChebPoly(\mathbf{q})$,  $\coeff_{\mathbf{e}_i} = \frac{\partial \ChebPoly}{\partial x_i}(\mathbf{r}_i)$, and $\coeff_\k  = \left(\prod_{i=1}^n \frac{1}{k_i! B_{k_i}}\right) \frac{\partial^{k_1}}{\partial x_1^{k_1}} \dots \frac{\partial^{k_n}}{\partial x_n^{k_n}} \ChebPoly(\mathbf{s_k})$, where, since these are Chebyshev polynomials, $B_{k_i} = 2^{k_i -1}$ unless $k_i = 0$, in which case $B_0 = 1$.
\end{corollary}

\begin{definition}
    For each Chebyshev polynomial approximation $\FuncApprox_j$ of the system $\{\Func_i\}_{i=1}^n$ of functions we want to find zeros of, and for each multiindex $\k\in \N^n$ let 
    \[ \MaxFuncDeriv_{j,\k} = \underset{{\x} \in [-1,1]^n}{\max} \prod_{i=1}^n  \left( \frac{1}{k_i! B_{k_i}}\right) \frac{\partial^{k_1}}{\partial x_1^{k_1}} \dots \frac{\partial^{k_n}}{\partial x_n^{k_n}} \FuncApprox_j({\x}) .
    \]
\end{definition}
Using this notation, we now have a bound on the $\k$th coefficient of $\FuncApprox_j$ 
\[
\abs{\coeff_{\k}} \leq \MaxFuncDeriv_{j,\k}.
\]

\subsection{Quadratic convergence}\label{sec:R-quadtratic-proof}

With these results in hand, we are now prepared to prove Theorem~\ref{thm:quadratic}
To start, we define some notation, which we use in continuation with that of Definition~\ref{def: Notation Checks}.

\begin{definition}\label{def:quadratic-notation}
\hfill
    \begin{itemize}
        \item $\RatioReduc$ : Let the width of the interval $J_i$ (without the consecutive linear rescalings back to the standard interval $I$ at each step) at the start of iteration $i$  in coordinate $m$ be $\gamma_{i,m}$, and define $\RatioReduc_{i,m} = \frac{\gamma_{i,m}}{\gamma_{i-1,m}}$, that is, the factor by which the interval shrinks in coordinate $m$ when moving from $J_{i-1}$ to $J_{i}$ .
        
        \item $\TotalReduc$ : Let $\TotalReduc_{i,m} = \prod_{\ell \le i}\RatioReduc_{\ell,m} = \frac{\gamma_{i,m}}{\gamma_{0,m}}$. At iteration $i$ this is the total shrinkage in coordinate $m$.  Let $\TotalReduc_i$ = $\max_m \TotalReduc_{i,m}$.
        Note that the total error $e_i = \max_{\x\in J_i} \|\x - \z\|$ in  interval $J_i$ satisfies \begin{equation}\label{eq:interval-bounded-by-sigma}
        e_i \le \sqrt{\sum_{m=1}^n \gamma_{i,m}^2} \le \sqrt{n} \max_{m} \gamma_{i,m} = c \sqrt{n} \TotalReduc_{i},
        \end{equation}
        for $c = \frac{1}{\min_{m} \gamma_{0,m}}.$
        Therefore, to prove R-quadratic convergence of the intervals, it suffices to show that $(\TotalReduc_i)_{i=1}^\infty$ converges quadratically to zero.
        
        \item $\MinDetDeriv$:  Let $\MinDetDeriv_i = \min_{\DerivMatrix \in {\AllDerivMatrices_i}}\abs{\det{(\DerivMatrix)}}$ where $\AllDerivMatrices_i$ is the set of all $\dimension \times \dimension$ matrices $D$ whose $j,m$ entry satisfies $\DerivMatrix_{jm} = \frac{\partial \FuncApprox_{j}}{\partial x_m} (c_{jm})$ for some $c_{jm}$ in interval $J_i$. Corollary~\ref{crly:Bound On All Degree Coeffs Cheb ND} shows that the matrix $A$ of linear Chebyshev coefficients (see Definition~\ref{def: Notation Checks})
        lies in $\AllDerivMatrices_i$, and thus $\MinDetDeriv_i$ is a lower bound on $\abs{\det{\LinMatrix}}$.  At a simple zero $\z$ of $\FuncApprox_1,\dots, \FuncApprox_n$, the determinant $\det\left.\left(\frac{\partial \FuncApprox_i}{\partial x_j}\right|_{\z}\right)$ of the Jacobian is nonzero, so continuity implies that  $\MinDetDeriv_i$ is nonzero, provided the current interval $J_i$ contains a simple zero of $\FuncApprox_1,\dots, \FuncApprox_n$ and is sufficiently small. 
 
        \item $\MaxFirstDeriv$: Let $\MaxFirstDeriv_{i,m} = \max_{j,{\x}} \abs{\frac{\partial}{\partial x_m}\FuncApprox_j({\x})}$, where the maximum is taken over all $\x\in J_i$. This is the maximum magnitude of the derivative in coordinate $m$ of any of the polynomial approximations  on the interval $J_i$  
        \item $\UglyProd$: Let $\UglyProd_{i,m} = (\dimension - 1)! \frac{\prod_{\ell \ne m} \MaxFirstDeriv_{\ell,i}}{\MinDetDeriv_i} \sum_{j=1}^n \sum_{\norm{\k}_1 > 1} \MaxFuncDeriv_{j,\k}$, and let $\UglyProd_i = \max_m \UglyProd_{i,m}$. This is the convergence factor  for polynomials with no approximation error. 
        \item \hypertarget{def:basecaseerror}{} $\BaseCaseError$: Let $\BaseCaseError_{i,m} = (\dimension - 1)! \frac{\prod_{\ell \ne m} \MaxFirstDeriv_{i,\ell}}{\MinDetDeriv_i} \sum_{j=1}^n \ApproxError_j$, and let  $\BaseCaseError_i = \max_m \BaseCaseError_{i,m}$. This is the contribution to the convergence factor of the approximation errors.
    \end{itemize}
\end{definition}

We can now bound how much the interval will shrink at a given iteration.
\begin{lemma} For each $m \in \{1,\dots, n\}$ and each $i\in \N$ we have
\[
\RatioReduc_{i+1,m} \le \UglyProd_{i,m} + \BaseCaseError_{i,m}.
\]
\label{alpha_le}
\end{lemma}

\begin{proof}
At step $i$ the algorithm gives $\RatioReduc_{i+1,m} \le \sum_{j=1}^n \abs{\LinMatrix^{-1}_{m,j}}\LinApproxErr_j$. Cramer's rule gives  $\RatioReduc_{i+1,m} \le \sum_{j=1}^n \abs{\frac{\adj(\LinMatrix)_{m,j}}{\det{\LinMatrix}}}\LinApproxErr_j$. Corollary~\ref{crly:Bound On All Degree Coeffs Cheb ND} implies that the magnitude of each element of column $m$ of $\LinMatrix$ is bounded by  $\MaxFirstDeriv_{i,m}$, so by cofactor expansion, $\adj(A)_{m,j} \le (n-1)!\prod_{\ell \ne m}\MaxFirstDeriv_{i,\ell}$. Thus $\RatioReduc_{i+1,m} \le \sum_{j=1}^n (n-1)!\frac{\prod_{\ell \ne m}\MaxFirstDeriv_{i,\ell}}{\MinDetDeriv_i}\LinApproxErr_j \le \UglyProd_{i,m} + \BaseCaseError_{i,m}$, using $E_j =  \ApproxError_j + \sum_{\norm{\k} > 1} \MaxFuncDeriv_{j,\k}$.
\end{proof}

This can be used to bound how much an interval shrinks after $i$ iterations.
\begin{lemma}\label{lem:alphabound} For each $m\in\{1,\dots, n\}$ and each $i\in \mathbb{N}$ we have 
\[\RatioReduc_{i+1,m} \le \frac{\TotalReduc_{i}^2\UglyProd_{0,m}}{\TotalReduc_{i,m}}   + \BaseCaseError_{i,m} \le \frac{\TotalReduc_{i}^2\UglyProd_{0,m}   + \BaseCaseError_{0,m}}{\TotalReduc_{i,m}}.
\]
\end{lemma}
Note that the $\TotalReduc^2_i$ term of the numerator involves not $\TotalReduc_{i,m}$, but rather  $\TotalReduc_i = \max_{m} \TotalReduc_{i,m}$. 
\begin{proof}
The interval rescaling that transforms $J_i$ to to the standard interval $I$ means that the polynomials $\FuncApprox_j$ are all rescaled by $\RatioReduc_{i,m}$ in coordinate $m$, so the chain rule guarantees that each partial derivative in coordinate $m$ will be scaled by $\RatioReduc_{i,m}$.  The fact that $J_{i} \subseteq J_0 = I$ means that  $\AllDerivMatrices_{i}\subseteq \AllDerivMatrices_0$, which implies that        \begin{equation}\label{eq:MinDetDeriv}
            \MinDetDeriv_{i}\ge \left(\prod_{m=1}^n \TotalReduc_{i,m} \right)\MinDetDeriv_0
        \end{equation}
        for every $i\in \N$ and every $m\in \{1,\dots, n\}$.
Similarly, we have \begin{equation}\label{eq:MaxFirstDeriv}
            \MaxFirstDeriv_{i,m} \le \TotalReduc_{i,m} \MaxFirstDeriv_{0,m}  
        \end{equation}
        for every $i\in \N$ and $m\in \{1,\dots, n\}$.

Moreover, if $\MaxFuncDeriv_{i,(j,\k)}$ denotes the value of $\MaxFuncDeriv_{j,\k}$ at step $i$, then transforming $J_i$ to $J_0=I$ gives 
\begin{equation}\label{eq:MaxFuncDeriv}
\sum_{\norm{\k}_1 > 1} \MaxFuncDeriv_{i, (j,\k)}
\le \TotalReduc_i^2 \sum_{\norm{\k}_1 > 1} \MaxFuncDeriv_{0, (j,\k)},
\end{equation}
where $\sigma_i = \max_{m} \sigma_{i,m}$
Combining Equations~\eqref{eq:MinDetDeriv}, \eqref{eq:MaxFirstDeriv}, and \eqref{eq:MaxFuncDeriv}
gives 
\begin{equation} \label{eq:UglyProd}
    \UglyProd_{i,m} \le \frac{\TotalReduc_i^2}{\TotalReduc_{i,m}}
    \UglyProd_{0,m}
\dsand 
 \BaseCaseError_{i,m} \le 
    \frac{\BaseCaseError_{0,m}}{\TotalReduc_{i,m}}
\end{equation}
        
Putting these together gives  
\[
\RatioReduc_{i+1,m} \le \UglyProd_{i,m} + \BaseCaseError_{i,m}  \le \frac{\TotalReduc_{i}^2\UglyProd_{0,m}}{\TotalReduc_{i,m}}   + \BaseCaseError_{i,m} \le  \frac{\UglyProd_{0,m}  \TotalReduc_i^2 + \BaseCaseError_{0,m}}{\TotalReduc_{i,m}},
\]
as required.
\end{proof}

We can now bound the total size of the interval after any number of steps. 
\begin{corollary}\label{crly: ND Linear Check Convergence}
The following holds:
\begin{align}     
\TotalReduc_{i+1,m} &\le \UglyProd_0 \TotalReduc_i^2  +  \TotalReduc_{i,m} \BaseCaseError_{i,m}
 \le 
 \UglyProd_0 \TotalReduc_i^2 
+ \BaseCaseError_0 \label{eq:quadratic-plus}  
\end{align}
Moreover the sequence $(\sigma_{i,m} \BaseCaseError_{i,m})_{i=0}^\infty$ is decreasing in $i$.
\end{corollary}
\begin{proof}
Multiply the result of Lemma~\ref{lem:alphabound} by 
$\TotalReduc_{i,m}$, and use  $\TotalReduc_{i+1,m} = \RatioReduc_{i+1,m}\TotalReduc_{i,m}$.  
The same argument as used for \eqref{eq:UglyProd} shows that $ \BaseCaseError_{i,m}\TotalReduc_{i,m} \le \BaseCaseError_{i-1,m} \TotalReduc_{i-1,m}$ for all $i$, showing that the sequence  $(\sigma_{i,m} \BaseCaseError_{i,m})_{i=0}^\infty$ is decreasing.
\end{proof}
We now have all the pieces we need to finish the proof of quadratic convergence.
\begin{proof}(of Theorem~\ref{thm:quadratic})
If the approximation errors $\ApproxError_1,\dots, \ApproxError_n$ are all $0$, then $\BaseCaseError_{i,m} = 0$ for all $i$ and $m$, and thus Corollary~\ref{crly: ND Linear Check Convergence} implies $(\TotalReduc_i)_{i=0}^\infty$ 
converges quadratically to $0$, and hence Equation~\eqref{eq:interval-bounded-by-sigma} guarantees that $J_0\supseteq J_1\supseteq \dots$ converges R-quadratically to the unique zero $\z$.

In the case that $\BaseCaseError_{i,m}>0$, multiplying \eqref{eq:quadratic-plus} by $\gamma_{0,m}$ gives the relation \eqref{eq:almost-quadratic-sides} with $K=\UglyProd_0 \gamma_{0,m}$. The fact that $(\sigma_{i,m} \BaseCaseError_{i,m})_{i=0}^\infty$ is decreasing was established in Corollary~\ref{crly: ND Linear Check Convergence}, so as $\sigma_{i,m}\gamma_{0,m} = \gamma_{i,m}$, $(\gamma_{i,m} \BaseCaseError_{i,m})_{i=0}^\infty$ is also decreasing.
This completes the proof of Theorem~\ref{thm:quadratic}.

\end{proof}

Because we generally have $\BaseCaseError > 0$, we need to know approximately how small to expect repeated applications of the 
reduction method to shrink the interval, which tells us when to stop zooming in (this is the base case of the recursion described in Section~\ref{sec:recursion}).
The following theorem motivates our choice of the test to identify the base case.

\begin{theorem} \label{thm: ND Linear Check Convergence Error Bound}
The limit  $\TotalReduc = 
\lim_{i \to \infty} \TotalReduc_i$ exists and \[
\TotalReduc \le \BaseCaseError_0 \sum_{n=0}^\infty \frac{(\UglyProd \BaseCaseError_0)^n}{n+1} \binom{2n}{n} < 2.5 \BaseCaseError_0.
\]
\end{theorem}
\begin{proof}
For convenience, let $\BaseCaseError = \BaseCaseError_0$.  The sequence $(\TotalReduc_i)_{i=0}^{\infty}$ is monotonically decreasing and bounded below by $0$, so it must converge to some $\TotalReduc$. Corollary~\ref{crly: ND Linear Check Convergence} shows that the maximum possible value of $\TotalReduc$ will satisfy 
\begin{equation*}
    \label{eq:fixed-points}
\TotalReduc = \UglyProd \TotalReduc^2 + \BaseCaseError,
\end{equation*}
otherwise the algorithm  would keep shrinking. This is a quadratic map with fixed points $\frac{1 \pm \sqrt{1-4\UglyProd \BaseCaseError}}{2\UglyProd}$. The map is decreasing between the fixed points and increasing outside of them, so the bound will be at the smaller point $\frac{1 - \sqrt{1-4\UglyProd \BaseCaseError}}{2\UglyProd}$, if this point is real.

Assuming we are in an interval where we can guarantee convergence, we have $\UglyProd + \BaseCaseError < 1, $ and thus $\UglyProd^2 + 2\UglyProd \BaseCaseError + \BaseCaseError^2 < 1$. Moreover, we have $(\psi - \BaseCaseError)^2 \ge 0$ so $\UglyProd^2 - 2\UglyProd \BaseCaseError + \BaseCaseError^2 \ge 0$. Combining these gives  $4\UglyProd \BaseCaseError < 1$. Thus the fixed points are indeed real.

By the generalized binomial theorem, the Taylor series of $\sqrt{1+x}$ is \begin{equation}\label{eq:binomial}
\sum_{n=0}^{\infty} \binom{\frac{1}{2}}{n} x^n = 1 + \sum_{n=1}^{\infty}\frac{(-1)^{n-1}x^n}{n 2^{2n-1}} \binom{2n-2}{n-1},
\end{equation}
which converges when $\abs{x} < 1$. Plugging $x = -4\UglyProd \bar{\epsilon}$
into~\eqref{eq:binomial}
gives
\[
\frac{1-\sqrt{1-4\UglyProd \BaseCaseError}}{2\UglyProd} = 
\frac{1}{2\UglyProd}\sum_{n=1}^{\infty}\frac{(4\UglyProd \BaseCaseError)^n}{n 2^{2n-1}} \binom{2n-2}{n-1} = 
\BaseCaseError \sum_{n=0}^\infty \frac{(\UglyProd \BaseCaseError)^n}{n+1} \binom{2n}{n}.
\]
We know that $\UglyProd \BaseCaseError < \frac{1}{4}$, and an extension of Sterling's Approximation \cite{Sterling} for $n > 0$ gives ${\sqrt {2\pi n}}\left({\frac {n}{e}}\right)^{n}e^{\frac {1}{12n+1}} < n! < {\sqrt {2\pi n}}\left({\frac {n}{e}}\right)^{n}e^{\frac {1}{12n}}$. Thus we get
\[
\BaseCaseError \sum_{n=0}^\infty \frac{(\UglyProd \BaseCaseError)^n}{n+1} \binom{2n}{n} \le
\BaseCaseError (1 + \sum_{n=1}^\infty \frac{\sqrt {4\pi n} (\frac{2n}{e})^{2n} e^{\frac{1}{24n}}}{4^n(n+1)(\sqrt {2\pi n} (\frac{n}{e})^{n} e^{\frac{1}{12n+1}})^2}) \le
 (\BaseCaseError + \sum_{n=1}^\infty \frac{\BaseCaseError}{n^{\frac{3}{2}}\sqrt{\pi}})
\]

Hence $\TotalReduc \le \BaseCaseError(1 + \frac{\zeta(1.5)}{\sqrt{\pi}}) \approx 2.474\BaseCaseError$, where $\zeta$ is the Riemann zeta function.
\end{proof}

When checking the base case we set all the terms of degree at least 2 to $0$, which then makes $\UglyProd = 0$ because it is an empty sum. By Lemma~\ref{alpha_le}, this implies $\IntWidth \leq \BaseCaseError$.
If the interval scaling $\RatioReduc$ is greater than $\frac{1}{2.5} = 0.4$, we stop because that would force $\BaseCaseError \geq 0.4$, which could permit $\TotalReduc > 1$.

As long as we have not reached the base case, if the reduction method does not zoom in, subdivision will decrease the term $\UglyProd_m$ until reduction does zoom in. The main condition needed for that is that $\MinDetDeriv > 0$, which must hold in a sufficiently small neighborhood of a simple zero.

It should be noted that the proofs in this section can be extended to other orthogonal bases of polynomials and to the usual basis of monomials, as the arguments rely on the connection between the coefficients and the derivatives given in Theorem~\ref{thm:Bound On All Degree Coeffs ND}.

\subsection{Behavior at Multiple Roots}

The problem of finding zeros that are not simple (multiple) is an ill-conditioned problem, meaning that tiny changes in the inputs can result in massive changes to the outputs.  Thus no algorithm can be expected to do a consistently good job of 
finding all multiple or near-multiple roots.
But some solvers fail more severely when applied to multiple roots. Our algorithm still correctly bounds these roots and makes progress towards the correct answer.   Quadratic convergence is no longer guaranteed, but all subintervals that are eliminated are correctly discarded. Although the reduction steps can still be applied near a multiple root, they generally will not shrink the interval much and thus our algorithm will likely subdivide to the base case (with linear, instead of quadratic convergence).  If the multiplicity of the root is very high, the subdivision could go too deep and exceed the memory capacity of the machine.  Our implementation throws a warning if the recursion goes too deep before roots were found. 

Although the quadratic convergence  guarantees no longer hold for multiple (or near-multiple) roots, in numerical tests our solver \emph{YRoots} seems to find multiple and near-multiple roots reasonably rapidly and accurately, without significant difficulty, including on the famously challenging near-multiple  \emph{devastating examples} of Noferini and Townsend \cite{Noferini, OurMacaulayPaper}.  
In our tests if near-multiple roots were too close together, then \emph{YRoots} sometimes found a single small box enclosing more than one.  See Section~\ref{sec:multiple-numerical} for more details about numerical tests of multiple and near-multiple roots.

\section{Arithmetic Complexity} \label{section: complexity}

Here we present an analysis of the arithmetic complexity of our algorithm. For simplicity of the analysis, we assume that we are solving a system of $n$ functions in $n$ dimensions, each of which is approximated by a polynomial of degree $(d-1)$ in each dimension, and thus represented by a tensor of dimension $n$ and size $d$ in each dimension. It is straightforward to extend the following analysis to polynomials with differing degrees in each dimension.

The first part of our solver, the Chebyshev approximator, is easy to analyze. Because we double the degree at each step of the degree search, its complexity is dominated by either the function evaluations or the final FFT, so is $O(n (d^n\log(d)) + F d^n)$ where $F$ is the cost of a single function evaluation. Which of these terms dominates will depend on $F$. 

In the rest of this section, we analyze the arithmetic complexity of just the Chebyshev polynomial solver. First we require some mathematical background.

\subsection{Tau} \label{subsection: Tau}

One of the main questions in analyzing the complexity of the algorithm is how the degree of a Chebyshev polynomial changes after subdivision. The subdivision happens one coordinate at a time so we consider this question in one dimension: for a given $\IntWidth$, $\IntCenter$, what is the numerical degree of $T_n(\IntWidth x + \IntCenter) = \sum_{k=0}^n \coeff_k T_k(x)$?  For the following analysis we assume $0 < \IntWidth < 1$ and $0 \le \IntCenter \le 1 - \IntWidth$. We can ignore $\IntCenter < 0$ because, by symmetry, $T_n(\IntWidth x + \IntCenter)$ has the same degree as $T_n(\IntWidth x - \IntCenter)$.  Following the notation in \cite{Townsend}, we use $\tau$ to denote the rate at which the degree drops when subdividing.
\begin{definition} \label{def: Tau}
Fix $\IntWidth, \IntCenter$. Let $\gamma > 0$. Define $g(\gamma, n)$ as the $\ceil{n\gamma}$th coefficient of the Chebyshev expansion of $T_n(\IntWidth x + \IntCenter)$. Define $\tau_{\IntWidth,\IntCenter} = \min \{\gamma | \lim_{n \to \infty} g(\gamma, n) = 0\}$, and $\tau_{\IntWidth} = \sup_{\IntCenter}(\tau_{\IntWidth,\IntCenter})$.
\end{definition}

Intuitively, for large $n$, the function  $T_n(\IntWidth x + \IntCenter)$ will be of degree $n\tau_{\IntWidth, \IntCenter}$, and scaling an interval by a factor of $\IntWidth$ will scale the degree by  $\tau_\IntWidth$ or smaller. We now prove some bounds on values of $\tau$. This analysis heavily involves Bernstein ellipses. 

\begin{definition}
    For $p > 1$, the \emph{Bernstein ellipse} $E_p$ is the ellipse with foci at $\pm 1$ with major axis length $\frac{1}{2}(p + p^{-1})$ and minor axis length $\frac{1}{2}(p - p^{-1})$. We use the notation $\IntWidth E_p + \IntCenter$ to mean the new ellipse obtained by scaling every point in $E_p$ by the transformation $z \rightarrow \IntWidth z +\IntCenter$.
\end{definition}
\begin{lemma} \label{Lemma: Coeff Bound from Ellipse}
Let $p \ge q > 1$ and assume that $\IntWidth E_p + \IntCenter$ lies within the ellipse  $E_q$. If $T_n(\IntWidth z + \IntCenter) = \sum_{k = 0}^n \coeff_k T_k(z)$, then $\coeff_k \le \frac{q^n + q^{-n}}{p^k}$.
\end{lemma}

\begin{proof}
It is known \cite{Bernstein} that if $\abs{f(z)} \le M(p)$ for $z \in E_p$, then $\coeff_k \le \frac{2M(p)}{p^k}$.
The \emph{Joukouwski transformation} is defined as $J(z) = \frac{z + z^{-1}}{2}$. If $w = e^{iz}$, then the following equality holds: \begin{equation}\label{eq:joukouwski} T_n(J(w)) = T_n(\cos(z)) = \cos(nz) = \frac{w^n + w^{-n}}{2}.
\end{equation}
As $T_n$ is holomorphic, the equality  \eqref{eq:joukouwski} must hold for all $w \in \mathbb{C}$. If $C_q$ is the circle of radius $q$ centered at the origin, then $J(C_q) = E_q$, and thus  $T_n(E_q) = \frac{C_p^n + C_q^{-n}}{2}$. This implies that for all $z$ contained within the ellipse $E_q$, there exists a real $w$ such that $z = J(re^{iw})$ for some $r\le q$, and thus we have 
\[
\abs{T_n(z)} = \abs*{\frac{(re^{iw})^n + (re^{iw})^{-n}}{2}} \le \abs*{\frac{(re^{iw})^n}{2}} + \abs*{\frac{(re^{iw})^{-n}}{2}} \le \frac{q^n + q^{-n}}{2}.
\]
By hypothesis, for any $z \in E_p$, the transform   $\IntWidth z + \IntCenter$ lies inside $E_q$, and this gives 
\[
\abs{T_n(\IntWidth z + \IntCenter)} \le \frac{1}{2}(q^{n} + q^{-n}).
\]
Thus we may use the bound $M(p) = \frac{q^{n} + q^{-n}}{2}$, which gives $\coeff_k \le \frac{q^n + q^{-n}}{p^k}$.
\end{proof}

\begin{lemma} \label{Lemma: Ellipse Intersection}
Let $E_q$ be the smallest Bernstein Ellipse that contains $\IntWidth E_p + \IntCenter$. If $\IntCenter = 0$, then  $q - q^{-1} = \IntWidth(p - p^{-1})$. If $\IntCenter = 1 - \IntWidth$, then $q + q^{-1} = \IntWidth(p + p^{-1}) + 2\IntCenter$.
\end{lemma}

This is equivalent to saying that $E_q$ intersects $\IntWidth E_p + \IntCenter$ on the $y$-axis if $\IntCenter = 0$ and on the $x$-axis if $\IntCenter = 1 - \IntWidth$. A full proof is given in \ref{Proof: Ellipse Intersection}.

\begin{lemma} \label{Lemma: Tau Inf Bound}
If $E_q$ is the smallest Bernstein ellipse that contains $\IntWidth E_p + \IntCenter$, then $\tau_{\IntWidth,\IntCenter} \le \inf_{p>1} \frac{\log{(q})}{\log{(p)}}$.
\end{lemma}

\begin{proof}
Lemma~\ref{Lemma: Coeff Bound from Ellipse} implies that $\coeff_{\ceil{n\gamma}} \le \frac{q^n+q^{-n}}{p^{\ceil{n\gamma}}}$. This goes to $0$ if and only if $p^{\gamma} > q$. The result follows.
\end{proof}

\begin{theorem} \label{thm: Tau Bound}
$\tau_{\IntWidth, 0} \le \IntWidth$ and $\tau_{\IntWidth, 1 - \IntWidth} \le \sqrt{\IntWidth}$.
\end{theorem}

This follows from applying the result of Lemma~\ref{Lemma: Ellipse Intersection} to the infimum from Lemma~\ref{Lemma: Tau Inf Bound}. Full details are given in \ref{Proof: Tau Bound}.

\begin{corollary}\label{cor:tau-bound}
$\tau_{\IntWidth} \le \sqrt{\IntWidth}$.
\end{corollary}
\begin{proof}
For fixed $\IntWidth$, as $\IntCenter$ increases,  $q$ and, therefore, $\tau_{\IntWidth,\IntCenter}$ increases. Thus $\tau_{\IntWidth, \IntCenter}$ is maximized at $\IntCenter = 1-\IntWidth$, so $\tau_{\IntWidth} \le \sqrt{\IntWidth}$.
\end{proof}

It is currently unknown how to find an explicit value of $q$ in general, and so a general solution for $\tau_{\IntWidth,\IntCenter}$ is unknown.

\subsection{Tau numerical testing}
We can estimate the value of $\tau$ numerically by computing the numerical degree of $T_n(\IntWidth x + \IntCenter)$. If $C_{i,j}(\IntWidth, \IntCenter)$ is the Chebyshev transformation matrix, the new degree is $D = \max_i : \abs{C_{i,n}(\IntWidth, \IntCenter)} > \varepsilon_{machine}$. For large $n$, we can then approximate $\tau_{\IntWidth,\IntCenter} \approx \frac{D}{n}$. The results of doing this for $\alpha =0.5$ are plotted in Figure~\ref{fig:Tau Convergence Plot}.  These computations and others  motivate the following conjecture for the value of  $\tau_{\IntWidth, \IntCenter}$.

\begin{conjecture} \label{Conj:Actual Tau}
$$\tau_{\IntWidth,\IntCenter} = \frac{1}{(\frac{1}{\IntWidth} - \frac{1}{\sqrt{\IntWidth}})\sqrt{1-(\frac{\IntCenter}{1-\IntWidth})^2} + \frac{1}{\sqrt{\IntWidth}}}$$
\end{conjecture}

Note that for fixed $\alpha$, the graph of $\frac{1}{\tau_{\IntWidth,\IntCenter}}$ is conjectured to be half an ellipse.

\begin{figure}[ht]
\centering
\includegraphics{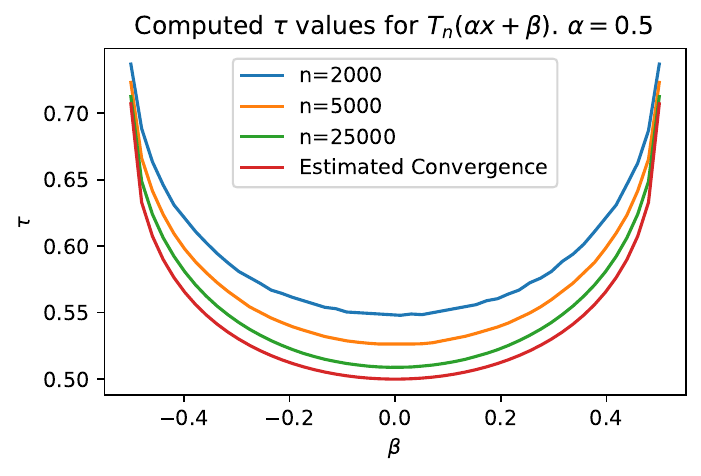}
\caption{The computed values of $\tau_{0.5,\IntCenter}$ for large degree $n$ along with the values from Conjecture~\ref{Conj:Actual Tau}, which the approximations seem to converge to.}
\label{fig:Tau Convergence Plot}
\end{figure}

\subsection{Arithmetic Complexity}\label{sec:temporal-complexity}
Now that we have results about how the degree changes as we subdivide, we can talk about the arithmetic complexity of the Chebyshev polynomial solver as a whole.

Here we just analyze the core algorithm of the repeated linear reduction method and polynomial transformations, with the understanding that any other reduction and elimination checks (potentially beyond what we describe in this paper), are only added to speed things up. We do this because their contribution is difficult to analyze because it is hard to say how often the elimination checks will throw out an interval, although that would be interesting future work.

\begin{theorem} \label{thm:TemporalComplexity}
Finding the common zeros of a system of $n$ polynomials of degree $d-1$ in each dimension has arithmetic complexity of $O(d^2\log{d})$ for $n=1$ and $O(nd^{2n})$ for $n>1$.
\end{theorem}
\begin{proof}
The linear reduction method requires summing all the terms in the coefficients and solving an $n \times n$ system, with a complexity of  $O(nd^n + n^3) = O(nd^n)$. The cost of reapproximating all the polynomials systems is $O(nd^{n+1})$. Thus the complexity of one full step of the algorithm is $O(nd^{n+1})$.

We assume the worst case where the algorithm splits the interval in half at every step, keeping all the intervals until the degree is very small, at which point it zooms in on all the zeros. After the first subdivision step the algorithm always splits the interval exactly in half, so we use $\tau$ to mean $\tau_{0.5,0.5}$. The cost of the next step is then $n2^n(\tau d)^{n+1}$. The ratio of the cost of each step to the cost of the previous step is thus $\frac{n2^n(\tau d)^{n+1}}{nd^{n+1}} = 2^n\tau^{n+1} \le \sqrt{2}^{n-1}$, by Corollary~\ref{cor:tau-bound}. The number of steps required will be roughly $2\log_2{d}$. For $n=1$, each step has the same cost so the total complexity is $O(d^2\log{d})$. For $n > 1$, the complexity is dominated by the final step, so is $n d^{n+1}(\sqrt{2}^{n-1})^{2\log_2{d}} = nd^{2n}$. So for $n>1$ we have complexity $O(nd^{2n})$.

This is the complexity to get to get down to a bunch of small intervals of small degree, to the point where the linear check is now starting to zoom in. At this point, if the system has $R$ common zeros, we assume we have to zoom in on roughly $R$ different intervals down to the base case. Because of the quadratic convergence of zooming in on the zeros, this will take a constant amount of time for each zero. For our given system we will have at most $(d-1)^n$ zeros. Thus the total complexity as is given above.
\end{proof}

\subsection{Arithmetic Complexity, Numerical Results}\label{sec:numerical-temporal-complexity}

In general the complexity will be better than that shown above because $\tau$ is often less than $\frac{\sqrt{2}}{2}$. Intervals on the interior have significantly lower degree than those on the edge. This is shown in Table~\ref{tab:ChebDegrees}, where we show the degree of $T_{10000}(x)$ on four rounds of subdivision (with $\alpha = 0.5$). Note that the intervals near $0$ scale by roughly $\IntWidth$, as would be expected from  Theorem~\ref{thm: Tau Bound}.

\begin{table}[ht] 
    \centering
    \begin{tabular}{|c|c|c|c|c|c|c|c|}
        \hline
        \multicolumn{8}{|c|}{7173} \\
        \hline
        \multicolumn{4}{|c|}{2730} & \multicolumn{4}{|c|}{5104}\\
        \hline
        \multicolumn{2}{|c|}{1375} & \multicolumn{2}{|c|}{1469} & \multicolumn{2}{|c|}{1738} & \multicolumn{2}{|c|}{3633}\\
        \hline
        718 & 728 & 750 & 791 & 853 & 964 & 1194 & 2589\\
        \hline
    \end{tabular}
    \caption{The $i$th row of this table gives the degree when approximating $T_{10000}$ on the $2^{i-1}$ evenly spaced intervals on $[0,1]$.  The leftmost columns are the intervals closest to $0$, the rightmost those closest to $1$. The degrees on $[-1,0]$ are the same (in reversed order to make them symmetric about $0$).  The degree of the approximations near $0$ (the left side of this table) scale approximately like $\alpha = \frac12$, as would be expected from  Theorem~\ref{thm: Tau Bound}. The degree on the intervals nearest $1$ scale more like our worst-case bound of $\sqrt{\alpha}$.}
    \label{tab:ChebDegrees}
\end{table}

\begin{table}[ht]     \centering
    \begin{tabular}{|c|c|c|c|c|c|c|}
        \hline
        Dim & Step 2 & Step 3 & Step 4 & Step 5 & Step 6 \\
        \hline
        1 & 1. & 0.63246763 & 0.36532202 & 0.20318345 & 0.11092983 \\
        \hline
        2 & 1.41421356 & 0.9580105  & 0.56098545 & 0.31115873 & 0.16892276\\
        \hline
        3 & 2. & 1.45111635 & 0.86144459 & 0.47651401 & 0.25723376\\
        \hline
        4 & 2.82842712 & 2.19803297 & 1.32282715 & 0.72974201 & 0.39171281\\
        \hline
    \end{tabular}
    \caption{The relative complexities of the first steps of solver relative to the first step.}    \label{tab:ComplexityChange}
\end{table}

Because of this, the complexity should actually be much better than shown in \ref{thm:TemporalComplexity}.  Assuming the degrees actually drop as predicted by Conjecture~\ref{Conj:Actual Tau}, we get the following conjecture, which is supported by Table~\ref{tab:ComplexityChange}.

\begin{conjecture} \label{Conj:Actual Complexity}
For any dimension $n$, and sufficiently large degree $d$, for a system of $n$ degree-$d$ Chebyshev polynomials in $n$ variables, the cost of subdivision will be greatest on the second step, which will be $\sqrt{2}^{n-1}$ times the complexity of the first step.  The complexity will then decrease with each step at an increasing rate, becoming $50\%$ cheaper at each step in the limit. Thus the arithmetic complexity of the entire algorithm should be the complexity of the second step, which is $O(nd^{n+1}\sqrt{2}^{n-1})$.
\end{conjecture}

This conjectured complexity seems to more closely match the results of our experiments in Figure \ref{fig:ChebPolySolveTimes}. In one dimension our solver has a provable complexity no worse than $O(d^2 \log{d})$, but is most likely $O(d^2)$. As seen in the figure, up until degree about 10,000, the complexity is almost linear, that is $O(d)$. This makes sense, because for small degree we  expect the base case to dominate, in which case it would be $O(R)$, where $R$ is the number of zeros. In our dense examples, this becomes $O(d^n)$. 

In summary, Theorem~\ref{thm:TemporalComplexity} guarantees a complexity no worse than $O(d^2 log(d))$ for $n=1$ and $O(d^{2n})$ for $n>1$, but we suspect it is really $O(d^{n+1})$.  For smaller degrees and dimensions, we expect that it initially grows as $O(R) \le O(d^n)$.

For comparison, a common alternative method for solving a single one-dimensional Chebyshev polynomial is to find the eigenvalues of the colleague matrix, which has complexity $O(d^3)$.  Recently eigenvalue-based solvers have been developed in one dimension that are $O(d^2)$ \cite{serkh2021provably}.
In higher dimensions, the eigenvalue-based methods use the Macaulay matrix, and, as shown in \cite{OurMacaulayPaper}, these methods have a arithmetic complexity that is at best $O(d^{3n})$.

\section{Stability Analysis} \label{section: stability}
We do not present a rigorous treatment of the stability of this algorithm, but we can give reasoning and numerical evidence that suggest this algorithm can find roots with an extremely high level of accuracy. However, it is possible that roots could be missed due to numerical instability as discussed here. That could potentially be fixed in future work by proving stability of the construction of the Chebyshev Transformation Matrix and using interval arithmetic in the implementation.

There are two principle steps in the algorithm at which we have to worry about error being introduced: first, transforming polynomials to new intervals and, second, zooming in to new intervals. Both of these are very well behaved in numerical tests, as discussed below.

\subsection{Error of polynomial transformation}

Transforming the polynomials to a new interval only requires a tensor multiplication of the Chebyshev transformation matrix by the polynomial tensor. Assume that $C_{n,m}(\IntWidth, \IntCenter)$ can be computed with maximum (entrywise) error of at most $\varepsilon_C$, each coefficient of our polynomials is bounded in magnitude by $N$, and the dimension we are transforming in has degree $D$.  In this case the error introduced in the transformation is at most $\varepsilon_C N D$ for each coefficient. So if we know what $\varepsilon_C$ is, we can bound this error and add it to the tracked approximation error at each step.

\subsubsection{Stability of recurrence relation}

While we cannot currently prove any useful bounds on the error involved in creating the Chebyshev transformation matrix, there are a few reasons (listed below) to believe it should be well behaved numerically. Additionally, the results of numerical experiments seem to show that it is well behaved. 

First, the following theorem shows the terms in the matrix are bounded in magnitude by $2$, which suggests the recurrence relation in \eqref{recurr_rel} should be stable and any errors that are introduced should not grow.

\begin{theorem} \label{thm:Bound On Matrix Terms}
If $\abs{\IntWidth} + \abs{\IntCenter} \le 1$, then
$\abs{\ChebMatrix_{n,m}(\IntWidth, \IntCenter)} \le 
    \begin{cases} 
      1 & n = 0 \\
      2 & \text{otherwise} 
   \end{cases}
$
\end{theorem}
\begin{proof}
From theorem 3.1 in \cite{Trefethen_ApproximationTheory}
we have that
\[\int_{-1}^{1} \frac{f(x)T_k(x)}{\sqrt{1-x^2}} dx = a_k \ChebConst_k \text{ where } \ChebConst_k =            \begin{cases} 
      \pi & k = 0 \\
      \frac{\pi}{2} & \text{otherwise} 
   \end{cases}
\]

Substituting $x = cos(\theta)$ gives 

\begin{equation} \label{eq:Cheb Coeff Formula}
a_k = \int_{0}^{\pi} \frac{f(\cos(\theta))T_k(\cos(\theta))\sin(\theta)}{\gamma_k\sqrt{1-\cos^2(\theta)}} d\theta = \int_{0}^{\pi} \frac{f(\cos(\theta))\cos(k\theta)}{\gamma_k} d\theta
\end{equation}
Thus we have
$$\ChebMatrix_{n,m}(\IntWidth, \IntCenter) = \int_{0}^{\pi} \frac{T_m(\IntWidth \cos(\theta) + \IntCenter)\cos(n\theta)}{\gamma_n} d\theta$$
If $\abs{\IntWidth} + \abs{\IntCenter} \le 1$ the numerator is bounded in magnitude by 1, so
\[
\abs{\ChebMatrix_{n,m}(\IntWidth, \IntCenter)} \le \frac{\pi}{\ChebConst_n} = 
    \begin{cases}
      1 & n = 0 \\
      2 & \text{otherwise.}
   \end{cases}
   \]
\end{proof}

Also, it is easy to verify from \eqref{recurr_rel} that if $M_k$ is the sum of the entries in column $k$, that $M_{k+1} = 2(\IntWidth + \IntCenter) M_k - M_{k-1}$. This has the characteristic equation $\lambda^2 - 2(\IntWidth + \IntCenter) \lambda + 1 = 0$, with eigenvalues $\lambda = \IntWidth + \IntCenter \pm \sqrt{(\IntWidth + \IntCenter)^2 - 1}$. For $\IntWidth + \IntCenter \le 1$, $\abs{\lambda} = 1$. So if errors are introduced the sums of the errors in the columns will stay small. This does not guarantee the magnitude of the individual entries will stay small, but that seems likely. Finally, similarity can be seen between the recurrence relation and Clenshaw's algorithm, which is known to be stable \cite{Clenshaw}. We expect an analysis similar to that of Clenshaw's algorithm could be used to prove stability of the matrix creation.

We can also numerically estimate the error of creating the Chebyshev transformation matrix by creating it with many extra digits of precision and comparing to the standard double-precision result. Randomly choosing $1000$ values of $\alpha$ from $\operatorname{Uniform}(0,1)$ and $\beta$ from $\operatorname{Uniform}(0,\alpha)$, the maximum error observed in any entry of any matrix up to column $100$ is $\num{3.82e-15}$. And for $C(\frac{1}{2},\frac{1}{2})$ which is the matrix used for subdividing intervals, the error is $0$ until column 58 (because of the special nature of division by two in binary arithmetic).

\subsection{Error of zooming in on intervals}

A numerical analyst would rightly be suspicious that error might be introduced by zooming in on intervals. 
The stability of this part of the algorithm depends on three main ideas. First, that the matrix $\LinMatrix$ we invert has similar conditioning to the Jacobian matrix $J$. We must invert $\LinMatrix$ to determine the interval width at the next iteration. However, at each step, entries in $\LinMatrix$ are derivatives of our functions somewhere in the interval. Thus near the root $\LinMatrix$ should be close to the Jacobian matrix $J$.  So $\LinMatrix$ is expected to be reasonably well conditioned when $J$ is, which is approximately the conditioning of the root-finding problem itself. (See Section~\ref{section: convergence} for details on why $\LinMatrix$ behaves as it does.)  We also note that when we scale the interval by $\IntWidth$ it will scale the $i$th column of $J$ by $\IntWidth_i$, which could create an ill-conditioned matrix. We remedy this by preconditioning $\LinMatrix$ with column scaling. In our implementation we scale by an appropriate  power of $2$ (because the standard binary representation of floating-point numbers means this is an almost error-free computation)
to ensure that the maximum  value in every column of $A$ is in the interval $[\frac{1}{2},1)$.

Second, the accuracy of this part of the algorithm only depends on avoiding erroneously shrinking  intervals too much. In order to guarantee that we do not erroneously shrink an interval too much, we check the conditioning of $\LinMatrix$ as we solve, and if it is poorly conditioned we do not shrink further. We continue instead by subdividing and applying the exclusion checks to the subintervals. 

Finally, one might worry that about propagation of a small absolute error from the initial approximation or from trimming nearly-zero coefficients. Although initially insignificant, the error could grow as we zoom in and the values of $\LinMatrix$ shrink. But if at one iteration we scale our interval by $\IntWidth$, one expects the entries of $\LinMatrix$ to also scale by $\IntWidth$, and the relative error to increase by $\frac{1}{\IntWidth}$ on an interval of size $\IntWidth$.  Thus the contribution of this scaled error to the problem on the whole interval remains relatively constant. Since the algorithm converges quadratically, it will only have a few steps at which it can contribute any error at all.

\section{Numerical Tests} \label{section: numerical_tests}

In this section we present timing and accuracy results of two versions of our algorithm: 
\begin{enumerate}
    \item A Python implementation of the Chebyshev polynomial solver part of our method, without the  approximator. These results are presented in Subection~\ref{sec:numerical-poly-only}.
    \item A Python implementation of our combined algorithm (using both our Chebyshev approximator and our Chebyshev solver), which we call \emph{YRoots}.  These results are presented in Subection~\ref{sec:numerical-YRoots}.
\end{enumerate}

\subsection{Chebyshev Polynomial Solver}\label{sec:numerical-poly-only}

In this subsection we present timing and accuracy results of the Chebyshev polynomial solver part of our method, without the  approximator.  To do this we must start with polynomials already expressed in the Chebyshev basis. 

\subsubsection{Accuracy}

The Python implementation of our solver is more accurate than some of the built-in (standard library) functions in Python, and often gives results that are the best possible in double precision (the precision in which we have implemented this), meaning that that the zeros found by our algorithm are the closest possible floating point number to the true zero.

As an example to illustrate this, consider finding the zeros of $T_{1000}(x)$. This is an easy problem  analytically because we know the zeros are at 
$x = \cos{(\frac{k+0.5}{1000}\pi)}$ for $k = 0,1,\dots,999$. However, the numerical error of computing these cosine values of $x$ in Python (using the default double precision) is only accurate to within $\num{4e-16}$; whereas our algorithm computes all the zeros to within $\num{6e-17}$, and $943$ of them are the closest floating point value to the actual zero, so na\"ively comparing our solutions to the numerically computed values of $x = \cos{(\frac{k+0.5}{1000}\pi)}$ does not adequately reflect the accuracy of our results.

The way we have chosen to evaluate the accuracy of the zeros found with our solver is to Newton polish the zeros using $50$ digits of precision in Python's \texttt{mpmath} library. We then can report the distance from our found zero to the high-precision polished zero. But, when doing this, it is useful to remember that the best possible  solution is given by the double-precision floating point number closest to the true zero, which has an error not more than $2^{-54}$ ($\num{5.55e-17}$) for numbers between $0$ and $1$. For example, in Figure~\ref{fig:ChebErros1DMons} we plot (in blue) a histogram for the errors of the zeros found by our algorithm for every Chebyshev monomial $T_d(x)$ of degree $d=1$ to $d=1000$. On the $x$-axis is the size of the error, and on the $y$-axis is the density of found zeros with that error.  On the same plot, we plot (in black) the density of errors of the closest floating point numbers, that is, the distance from the true zero to the closest floating point number. These are very similar, and indeed $92.9\%$ of the zeros found by our algorithm were the closest floating point number to the true zero---that is, our computed zeros were the best possible numerical solution. The worst error for any of our computed zeros is $\num{1.5e-16}$.

\begin{center} \label{fig:ChebErros1DMons}
\begin{figure}[ht]
    \includegraphics{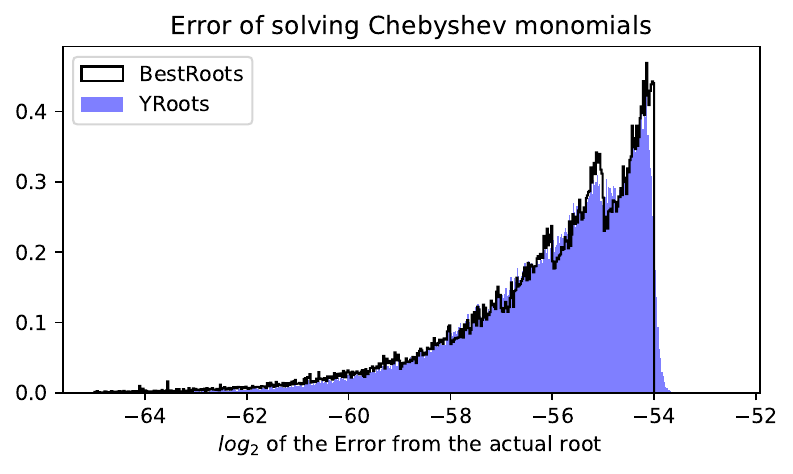}
      \caption{Histogram (in blue) of the errors (distance to the true value) for all of the zeros computed by our algorithm for the first $1000$ Chebyshev monomials $T_d(x)$.  The $x$-axis is the size of the error, and the $y$-axis is the density of the zeros with that error.  Plotted in black is the density of the errors of the best-possible numerical solutions (the nearest floating-point number to the true zero).}
\end{figure}
\end{center}

In higher dimensions our implementation of the polynomial solver (without the approximator) also has good accuracy finding the zeros of random Chebyshev polynomials. When solving the systems used in Figure \ref{fig:ChebPolySolveTimes}, the maximum error is $\num{1e-14}$, and the log average error is $\num{5e-17}$.

\FloatBarrier
\subsubsection{Timing}

Timing results of numerical tests of polynomial systems of increasing degree $\degree$ in dimensions $1$ through $5$, appears to be similar to the conjectured arithmetic complexity of $O(d^{n+1})$, as shown in Figure~\ref{fig:ChebPolySolveTimes}. In dimensions one through three it appears better, most likely because the degree is not big enough for the transformation step to dominate the complexity.

\begin{center} \begin{figure}[ht]
    \includegraphics{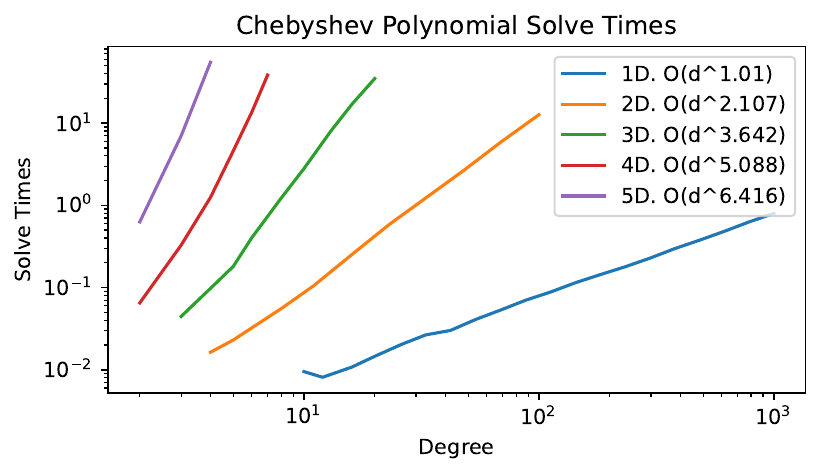}
      \caption{Average time for the Chebyshev polynomial solver to solve ten systems of Chebyshev-basis polynomials of varying total degree $\degree$ with coefficients drawn form the standard normal distribution in dimensions one through five.  The legend gives the observed arithmetic complexity of each, estimated from the slope of the plotted almost-straight lines. As functions of degree $\degree$ the smaller dimensional problems appear to be better even than the conjectured arithmetic complexity of $O(d^{n+1})$ (see Section~\ref{sec:numerical-temporal-complexity}).}
\label{fig:ChebPolySolveTimes}
\end{figure}
\end{center}

\FloatBarrier

\subsection{Full \emph{YRoots} Solver: Comparison to Other Methods}\label{sec:numerical-YRoots}

On a range of tests described below, 
we compared the speed and accuracy of a Python implementation of our combined algorithm (using both our Chebyshev approximator and our Chebyshev solver), which we call \emph{YRoots}, with the following solvers: \emph{Bertini} \cite{Bertini}, which uses a homotopy method; the eigenvalue-based solver of Mourrain, Telen, and Van Barel~\cite{Telen} (which we denote as {\emph{MTV}} in this section), implemented in Julia;  \emph{Chebfun2} (in two dimensions only), which uses Chebyshev proxy and subdivision, as we do, but then uses Bezout resultants to solve the systems on subintervals, and is implemented in MATLAB; and  Mathematica's \emph{Reduce} and \emph{NSolveValues}.  All tests were run on the same machine (a PowerEdge R640 with Intel Xeon Gold 6248 CPU server, 80 cores, and 768GiB RAM).

Our main tests were done on a collection of  randomly generated power-basis ($1,x,x^2,...$) polynomials in each dimension of varying degrees.  We also used the \emph{Chebfun Test Suite} \cite{Chebfun2Test} in two dimensions, which has some nonpolynomial systems; and we constructed some  (somewhat arbitrary) tests of nonpolynomial systems in three and four  dimensions.

\subsubsection{Random Polynomial Times}

In each dimension we generated $300$ polynomials for the random polynomial tests by drawing coefficients for the power-basis (standard monomials of the form $x_1^{k_1}\cdots x_n^{k_n}$) from the standard normal distribution, but setting the constant term to $0$ to ensure that there would be at least one zero in the standard interval $[-1,1]^n$. 

The timing results on the random polynomials are summarized in Figure~\ref{fig:dim23_timings}.
For higher degrees (more than degree ten in two dimensions and more than degree five in three dimensions) our \emph{YRoots} solver is substantially faster than all the other solvers. We expect that implementing \emph{YRoots} in a faster language like Julia or C would also make it much faster and competitive with the \emph{MTV} solver (written in Julia)  in those low-degree cases.   

\begin{figure}[ht]
\begin{center}
\includegraphics[width=\textwidth]{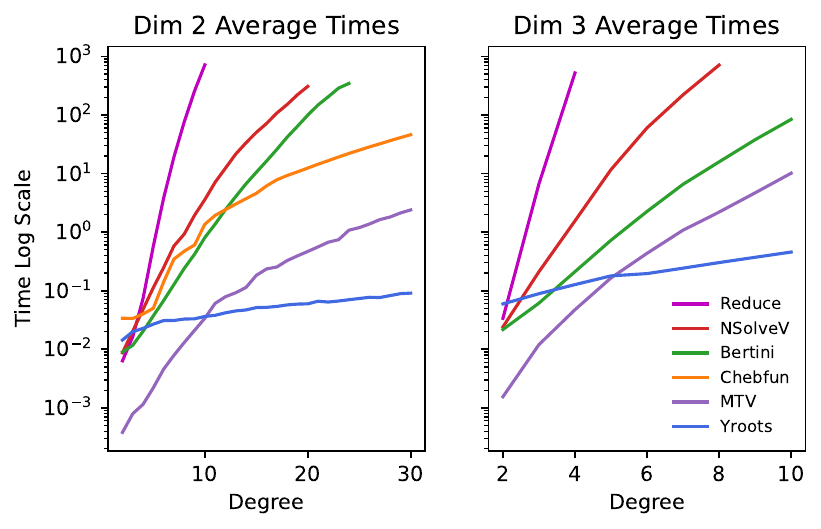}

      \caption{The results of running the random-polynomial tests in two (left panel) and three (right panel) dimensions, with degree on the $x$-axis and the timing in log scale on the $y$-axis. Each of the different methods is plotted as a different colored line.  \emph{Chebfun2} does not appear in the right panel because it is not yet implemented in dimensions greater than two.  We expect that implementing \emph{YRoots} in a faster language  like Julia or C would make it competitive with \emph{MTV} in low degree.}
      \label{fig:dim23_timings}
\end{center}
\end{figure}

\subsubsection{Avoids Undesired Zeros}

One reason our  \emph{YRoots} solver is faster than \emph{Bertini} and the \emph{MTV} solver is that \emph{YRoots} only solves for real zeros inside a given bounded interval (in these tests $[-1,1]^n$), while both of those other methods attempt to find all the zeros in $\C^n$.   \emph{Chebfun2} uses subdivision to reduce the degree of the Chebyshev approximation on each interval, which may allow it to avoid searching for some of the unwanted zeros of the original system, but it uses resultants to find the zeros of the approximation on each subinterval, and resultants also often find additional nonreal zeros (which are then discarded by \emph{Chebfun2}).

For the purposes of our testing, we ran \emph{NSolveValues} and \emph{Reduce} restricted to the standard interval, but they do have the capability to find zeros globally as well. 

The fact that \emph{MTV} and \emph{Bertini} (and, to some extent \emph{Chebfun2}), find more zeros than is needed probably contributes significantly to their completion time, especially in higher-dimensions and higher degree.   When the goal is to find real zeros in a bounded interval, then it is an advantage that our algorithm only finds such zeros and does not spend resources finding unwanted zeros.

\FloatBarrier
\subsubsection{Random Polynomial Accuracy}\label{sec:random-polynomial-results}

To evaluate the accuracy of the the computed zeros in the random polynomial tests, we use \emph{residuals}, which are the  values of the original functions (not the Chebyshev approximations) at the computed zeros.  If a computed zero is perfectly correct, then the residual should be zero.  For methods that also find zeros outside the real interval $[-1,1]^n$, the residuals were only computed for the zeros that do lie inside the interval.  

The maximum (worst) residuals for each solver in each degree in the polynomial tests in dimensions 2 and 3 are plotted in Figure~\ref{fig:residuals}.  In the higher-degree tests,  \emph{Reduce} failed to terminate in a reasonable amount of time (see Figure~\ref{fig:dim23_timings} for timings) and \emph{NSolveValues} terminated but failed to find all the zeros.  All the other methods found the same set of zeros (the same number of zeros and approximately the same locations for those zeros) for each system in the test set.  Moreover, since \emph{MTV} is guaranteed (within the limits of the stability of the eigenvalue solver) to find all of the zeros of each system, and \emph{YRoots} is guaranteed to find bounding intervals for all real zeros in the given interval, the fact that they and \emph{ChebFun2} and \emph{Bertini} all agreed on the number and location of the zeros gives us reason to believe that the only methods that missed any roots were \emph{Reduce} and \emph{NSolveValues}. 

Most of the solvers, including \emph{YRoots}, have consistently good results, but \emph{MTV} has residuals that are fairly consistently a factor of about 100 larger than most of the other methods, and \emph{NSolveValues} has residuals that are a factor of about $10^5$ larger than most of the methods.  This is probably due to the fact that MTV returns the results of a direct eigenvalue solver and has no subsequent refinement of its results, whereas Bertini and YRoots iteratively refine the results.  

\emph{Bertini} refines results with Newton's method to arbitrary precision.  \emph{YRoots} also refines roots by iteratively shrinking the containing interval for the approximating polynomials; but, unlike Newton's method, which is applied to the orriginal functions, the accuracy of the final \emph{YRoots} result is potentially limited by the accuracy of the Chebyshev approximation.  The fact that the \emph{YRoots} residuals are better than those of \emph{Bertini} shows that the Chebyshev approximation error is not having a significant negative effect on the final accuracy of the combined \emph{YRoots} solver.

\begin{figure}[ht]
\begin{center}    \includegraphics[width=\textwidth]{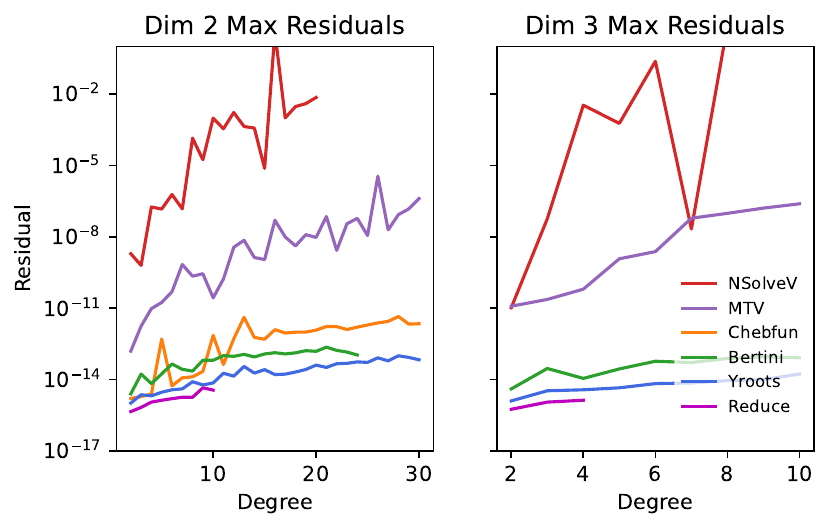}
\end{center}
      \caption{The maximum (worst) residuals for the solutions found by each of the the different solvers on the 2-dimensional (left panel) and 3-dimensional (right panel) random-polynomial tests, with degree on the $x$-axis and the residual in log scale on the $y$-axis. Each of the different methods plotted as a different colored line. The lines for \emph{Reduce} and \emph{NSolveValues} are only plotted for low degrees because in all the missing higher degrees they failed to terminate, or when they terminated they failed to find all of the roots.  Note that for all of these methods, if the final results are sufficiently close to the correct answers, then the computed roots could be Newton polished to arbitrary precision.  In such cases, the the practical impact of the differences in residual is minimal.  For more on this, see the discussion in Subsection~\ref{sec:random-polynomial-results}.}
      \label{fig:residuals}
 \end{figure}

\FloatBarrier
\subsubsection{Chebfun Test Suite and Higher-Dimensional Nonpolynomial Tests}

In addition to tests on polynomials, we ran timing and residual tests for the various solvers on the Chebfun Test Suite, which is a collection of purely two-dimensional zero-finding problems.  We also created some of our own nonpolynomial tests in dimensions three and four.

Although some of the Chebfun Test Suite involves systems of power-basis polynomials,  several Chebfun tests  involve nonpolynomial functions.  \emph{Bertini} and \emph{MTV} are not included in these results because they only work on polynomials expressed in the power basis. 

The results on the Chebfun Test Suite are plotted in Figure~\ref{fig:chebfun_times} (timing) and Figure~\ref{fig:chebfun_residuals} (residuals).  For most of tests, the solve times for \emph{YRoots} were roughly comparable to \emph{NSolveValues} and faster than \emph{Chebfun2} by a factor of about $10$.  Solve times for \emph{Reduce} were much more variable, sometimes substantially beating all competitors (Test 1.5), and sometimes much slower than all others (Test 8.2) and sometimes failing completely (Test 1.2). 
Residuals on the Chebfun Test Suite are mostly below $10^{-13}$ for all the solvers, with a few notable exceptions for each of the solvers.  

Only \emph{YRoots}, \emph{NSolveValues}, and \emph{Reduce} are able to solve higher-dimensional nonpolynomial problems.  On our higher-dimensional tests all three of these solvers mostly had good residuals and fairly similar solve times, except that there were some problems that \emph{YRoots} could solve correctly but \emph{Reduce} and \emph{NSolveValues} could not solve at all (missing zeros or did not terminate).

\begin{center}
\begin{figure}[ht]
    \includegraphics[width=\textwidth]{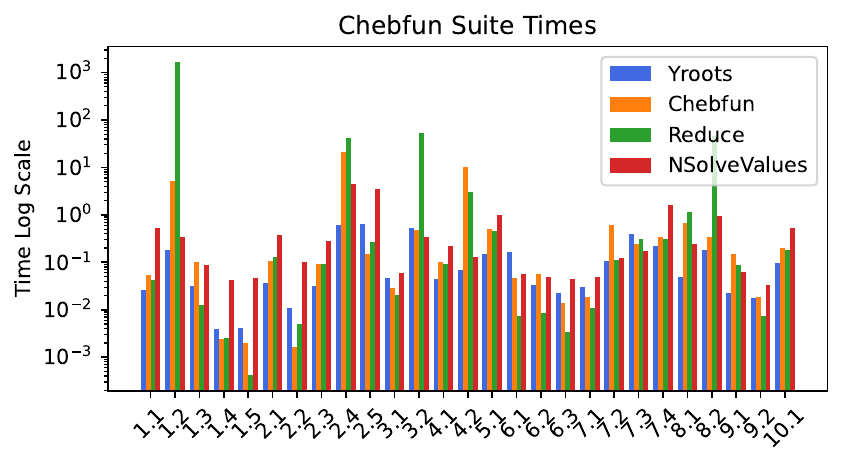}
      \caption{Time to solve each of the Chebfun Test Suite tests, using each of the solvers that work with nonpolynomial functions. \emph{Reduce} failed to solve Test 1.2 at all, so its time on that test is plotted as infinite.}
      \label{fig:chebfun_times}
 \end{figure}
\end{center}

\begin{center}
\begin{figure}[ht]
    \includegraphics[width=\textwidth]{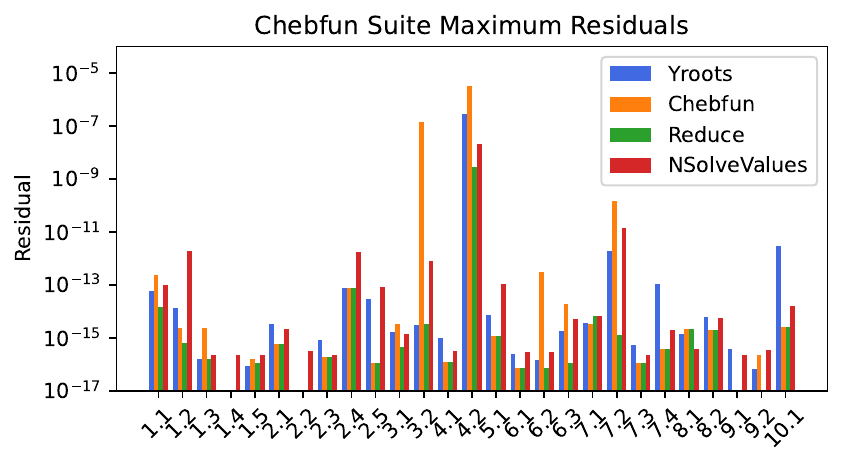}
      \caption{Maximum of the residuals on each of the Chebfun Test Suite tests, using each of the solvers that work with nonpolynomial functions.}
      \label{fig:chebfun_residuals}
 \end{figure}
\end{center}

\FloatBarrier
\subsubsection{Multiple and Near-multiple Zeros}
\label{sec:multiple-numerical}

Finding multiple and near-multiple zeros of a polynomial system is an ill conditioned problem, meaning that even tiny changes in the coefficients of the system result in large changes to the location of the root(s). Hence no numerical algorithm can be expected to solve these  well.  Nevertheless, we have performed some numerical tests comparing the performance of our implementation \emph{YRoots} with other standard solvers.

A particularly challenging collection of near-multiple zeros are the \emph{devastating examples} of Noferini and Townsend \cite{Noferini, OurMacaulayPaper}.  
These are all of the form 
\[
\begin{pmatrix}
    x_1^2\\
    x_2^2\\
    \vdots\\
    x_n^2
\end{pmatrix} + \varepsilon Q \begin{pmatrix}
    x_1\\
    x_2\\
    \vdots\\
    x_n
\end{pmatrix},
\]
for an orthonormal $Q$ and small values of $\varepsilon$.  We fixed a choice of $Q$ in each dimension from $2$ to $6$ and tested \emph{Yroots},\emph{Bertini}, \emph{MTV}, \emph{NSolveValues}, and \emph{Reduce} on this example for the near-multiple roots with  $\varepsilon \in \{10^{-2},10^{-3},\dots, 10^{-8}\}$ and for the true multiple root when $\varepsilon=0$.  
The results for the test in dimension $3$ are given in Figure~\ref{fig:DevEx3}.  
The correct number of real roots is $4$ for the particular example we tested in dimension $3$. For $\varepsilon \le 10^{-6}$ \emph{YRoots} returns fewer than $4$ roots but gives bounding boxes that contain all the roots---just some bounding boxes contain more than one root.  \emph{MTV} loses roots already at $\varepsilon = 10^{-3}$ and does not return bounding boxes, whereas \emph{Bertini} finds too many roots when $\varepsilon\le 10^{-7}$, including $8$ zeros instead of $1$ for the lone true multiple zero at the origin when $\varepsilon=0$.  Computation time of \emph{YRoots} on these is not much slower than most of the other solvers, and is much faster than \emph{Reduce}.  All the solvers have good residuals except \emph{MTV} whose max residuals are very large.

The results of our tests in other dimensions were qualitatively similar to this one in dimension $3$.
\begin{figure}
\includegraphics[width=\textwidth]{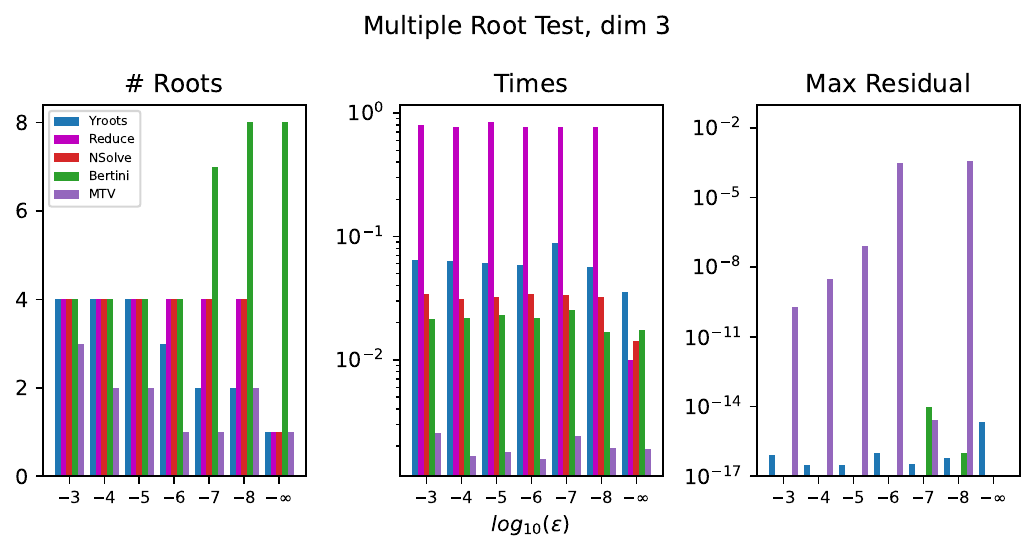}
\caption{Results of tests on the Noferini-Townsend devastating example for a fixed orthonormal matrix $Q$ and $\varepsilon \in \{10^{-3}, \dots, 10^{-8}, 0\}$. The true number of real zeros when $\varepsilon>0$ is $4$.}
\label{fig:DevEx3}
\end{figure}

\FloatBarrier
\section{Conclusion} \label{section: conclusion}
Our novel Chebyshev polynomial solver is able to find common zeros of a dense system of polynomials in the Chebyshev basis with extreme accuracy, with comparable or better speed than existing methods.  When combined with existing Chebyshev approximation methods, it can find common zeros of almost any sufficiently smooth system of equations.

There are some potential adjustments that could be made to further improve the algorithm. Better elimination checks or reduction methods could speed up the solver. Using the low-rank approximation methods of \emph{Chebfun2} might speed up solve time in higher dimensions. Finally, finding a way to do faster multiplication by the Chebyshev transformation matrix would also significantly increase the speed of the algorithm.

\appendix



%


\section{Proofs of Lemmas}

\subsection{Proof of Lemma~\ref{Lemma: Ellipse Intersection}}\label{Proof: Ellipse Intersection}
\begin{proof} 
For the $\IntCenter = 0$ case, we use the ellipse definition $\frac{x^2}{a^2} + \frac{y^2}{b^2} = 1$ where $2a$ and $2b$ are the width and height of the ellipse. Thus $y = \pm b\sqrt{1 - \frac{x^2}{a^2}}$, $\frac{dy}{dx} = \frac{-b\frac{2x}{a}}{\sqrt{a^2 - x^2}}$. So if we have two ellipses centered at $0$ with the same height, the one with the greater width will decrease slower, and thus fully contain the other one. So if $2b = q - q^{-1} = \IntWidth(p - p^{-1})$, then $E_q$ is contained in  $\IntWidth E_p$, and it intersects it so is the smallest such $q$.

For the $\IntCenter = 1 - \IntWidth$ case, note the if $q + q^{-1} = \IntWidth(p + p^{-1}) + 2\IntCenter$, then $E_q$ and $\IntWidth E_p + \IntCenter$ both have a foci of $1$. Call this foci $F$. Let the other foci of $E_q$ by $F_1$, and the other foci of $\IntWidth E_p + \IntCenter$ $F_2$. Let the intersection of the ellipses on the x axis be $P$. The sum of the distances from the foci to any point on the ellipse is constant for an ellipse. So for any point $Q$, $\abs{F_1Q} + \abs{FQ} = \abs{F_1P} + \abs{FP}$ and $\abs{F_2Q} + \abs{FQ} = \abs{F_2P} + \abs{FP}$. So $\abs{F_1Q} - \abs{F_2Q} = \abs{F_1P} - \abs{F_2P}$. As $F_1$, $F_2$ and $P$ are co-linear, $\abs{F_1Q} = \abs{F_1F_2} + \abs{F_2Q}$. So we get $\abs{F_1F_2} = \abs{F_1F_2} + \abs{F_2Q}$. By the triangle inequality this will only be true on the $x$-axis, and everywhere else $E_q$ will be outside of $\IntWidth E_p + \IntCenter$.
\end{proof}

\subsection{Proof of Theorem~\ref{thm: Tau Bound}} \label{Proof: Tau Bound}
\begin{proof}
Lemma~\ref{Lemma: Tau Inf Bound} gives $\tau_{\IntWidth,\IntCenter} \le \inf_{p>1} \frac{\log{(q})}{\log{(p)}}$. And Lemma~\ref{Lemma: Ellipse Intersection} implies that if $\IntCenter = 0$, then $t = \frac{1}{2}(q - q^{-1}) = \frac{\IntWidth}{2}(p - p^{-1})$ for $t > 0$. Thus  $q = t + \sqrt{t^2+1}$. $p = \frac{t}{\IntWidth} + \sqrt{(\frac{t}{\IntWidth})^2+1}$. $\tau_{\IntWidth,0} \le 
\inf_{t > 0}\frac{\log{(t + \sqrt{t^2+1})}}{\log{(\frac{t}{\IntWidth} + \sqrt{(\frac{t}{\IntWidth})^2+1})}} 
\le 
\lim_{t \to 0} \frac{\log{(t + \sqrt{t^2+1})}}{\log{(\frac{t}{\IntWidth} + \sqrt{(\frac{t}{\IntWidth})^2+1})}}
=
\lim_{t \to 0} \frac{\log{(t + 1 + \frac{t^2}{2})}}{\log{(\frac{t}{\IntWidth} + 1 + \frac{t^2}{2\IntWidth^2})}}
=
\lim_{t \to 0} \frac{\log{(1 + t)}}{\log{(1 + \frac{t}{\IntWidth})}}
=
\lim_{t \to 0} \frac{t}{\frac{t}{\IntWidth}}
= \IntWidth
$.

Lemma~\ref{Lemma: Ellipse Intersection} implies that  if $\IntCenter = 1-\IntWidth$, then $t = \frac{1}{2}(q + q^{-1}) = \frac{\IntWidth}{2}(p + p^{-1}) + \IntCenter$ for $t > 1$. So $q = t + \sqrt{t^2-1}$. $p = \frac{t-\IntCenter}{\IntWidth} + \sqrt{(\frac{t-\IntCenter}{\IntWidth})^2-1}$. $\tau_{\IntWidth,1-\IntWidth} \le 
\inf_{t > 1}\frac{\log{(t + \sqrt{t^2-1})}}{\log{(\frac{t-\IntCenter}{\IntWidth} + \sqrt{(\frac{t-\IntCenter}{\IntWidth})^2-1})}}$. Letting $t = 1 + \varepsilon$, gives $\tau_{\IntWidth,1-\IntWidth} \le 
\lim_{\varepsilon \to 0} \frac{\log{(1 + \varepsilon + \sqrt{2 \varepsilon + \varepsilon^2})}}{\log{(1 + \frac{\varepsilon}{\IntWidth} + \sqrt{2\frac{\varepsilon}{\IntWidth} + (\frac{\varepsilon}{\IntWidth})^2})}}
=
\lim_{\varepsilon \to 0} \frac{\log{(1 + \varepsilon + \sqrt{2\varepsilon})}}{\log{(1 + \frac{\varepsilon}{\IntWidth} + \sqrt{2\frac{\varepsilon}{\IntWidth}})}}
=
\lim_{\varepsilon \to 0} \frac{\varepsilon + \sqrt{2\varepsilon}}{\frac{\varepsilon}{\IntWidth} + \sqrt{2\frac{\varepsilon}{\IntWidth}}}
= \sqrt{\IntWidth}
$.
\end{proof}

\bibliographystyle{alpha} 
\newcommand{\etalchar}[1]{$^{#1}$}

\end{document}